\documentclass[]{interact}

\usepackage{epstopdf}
\usepackage[caption=false]{subfig}

\usepackage[numbers,sort&compress]{natbib}
\bibpunct[, ]{[}{]}{,}{n}{,}{,}

\theoremstyle{plain}
\newtheorem{theorem}{Theorem}[section]
\newtheorem{lemma}[theorem]{Lemma}

\theoremstyle{definition}
\newtheorem{definition}[theorem]{Definition}

\theoremstyle{remark}
\newtheorem{remark}{Remark}

\begin{document}


\title{ Existence and stability results for a coupled system of Hilfer fractional Langevin equation with non local integral boundary value conditions.}

\author{
\name{Khalid HILAL\textsuperscript{a}, Ahmed KAJOUNI\textsuperscript{a} and  Hamid LMOU\textsuperscript{a}\thanks{CONTACT Hamid LMOU. Email: hamid.lmou@usms.ma}}
\affil{\textsuperscript{a}Laboratory of Applied Mathematics and Scientific Competing, Faculty of Sciences and Technics,
Sultan Moulay Slimane University, Beni Mellal, Morocco}
}

\maketitle

\begin{abstract}

 This paper deals with the existence and uniqueness of solution for a coupled system of Hilfer fractional Langevin equation with non local integral boundary value conditions. The novelty of this work is that it is more general than the works based on the derivative of Caputo and Riemann-Liouville, because when $\beta =0$ we find the Riemann-Liouville fractional derivative and when $\beta =1$ we find the Caputo fractional derivative. Initially, we give some definitions and notions that will be used throughout the work, after that we will establish the existence and uniqueness results by employing the fixed point theorems. Finaly, we investigate different kinds of stability such as Ulam-Hyers stability, generalized Ulam-Hyers stability.
\end{abstract}

$2010$ MSC: $26A33, 39B82$

\section{Introduction}

 \quad The theory of fractional derivatives plays a very important role in modeling problems and describing natural phenomena in different fields such as physics, biology, finance, economics, and bioengineering \cite{ref2, ref5, ref6, ref12, ref14, ref17}. With the recent outstanding development in fractional differential equations, the Langevin equation has been considered a part of fractional calculus \cite{ref3, ref4, ref10}, an equation of the form $m \frac{d^2 x}{dt^2}=\lambda \frac{dx}{dt}+\eta (t)$ is called Langevin equation, introduced by Paul Langevin in 1908. The Langevin equation is found to be an effective tool to describe the evolution of physical phenomena in fluctuating environments \cite{ref18}. For some new developments on the fractional Langevin equation, see, for example, \cite{ref9, ref10} . \\
 
 In the literature, there are several definitions of fractional derivative , among them we find the one of Caputo and Riemann-Liouville. In \cite{ref11} Hilfer introduced the generalization of these two derivatives named by Hilfer fractional derivative of order $\alpha $ and type $\beta \in [0, 1]$, when $\beta =0$ we find the Riemann-Liouville fractional derivative and when $\beta =1$ we find the Caputo fractional derivative. Many papers have studied coupled systems for fractional differential equations, recently  S.K.Ntouyas \cite{ref16}, studied the existence of the solution for a coupled system given by 
 \begin{equation}
	\begin{cases}
	\displaystyle ^{H}D^{\alpha , \beta}x(t)= f(t,x(t),y(t)), \quad t\in [a, b],\\
	\displaystyle ^{H}D^{p , q} y(t)= g(t,x(t),y(t)), \quad t\in [a, b],\\
		x(a)=0 \quad , \quad  x(b)=\displaystyle \sum\limits_{i=1}^{n} \eta_i (I^{\chi_i}(y))  (\theta_i),\\
		y(a)=0 \quad , \quad  y(b)=\displaystyle \sum\limits_{j=1}^{m} \tau_j (I^{\varsigma_j}(x))  (\kappa_j).
	\end{cases}
\end{equation}
 \quad Where $ ^{H}D^{\alpha, \beta},^{H}D^{p, q} $, are the Hilfer fractional derivative of order $\alpha, p$, such that  $1<\alpha, p<2$ and parameter $\beta, q$, such that $0<\beta, q<1$, $a\geq 0$, $I^{\chi_i}, I^{\varsigma_j}$ are the Riemann-Liouville fractional integral of order $\chi_i$ and $\varsigma_j$ respectively $\chi_i, \varsigma_j > 0$, the points $\theta_i$, $\kappa_j$ $\in [a, b]$, $i=1,2,....,n$, $j=1,2,....,m$, $\eta_i, \tau_j \in \mathbb{R}$ and $f,g:[a, b]\times \mathbb{R}\times \mathbb{R} \longrightarrow \mathbb{R}$ are continuous functions. The author proved the existence and uniqueness results by using, Leray-schauder alternative, Krasnoselskii's fixed point theorem, and Banach' fixed point theorem.\\

 Motivated by the work above, we investigate the existence and uniqueness criteria for the solution of the following nonlocal  coupled system of Hilfer fractional Langevin equation:
  \begin{equation}
	\begin{cases}
	\displaystyle ^{H}D^{\alpha_1 , \beta_1} ( ^{H}D^{\alpha_2 , \beta_2}+\lambda_1  )x(t)= f(t,x(t),y(t)), \quad t\in [a, b],\\
	\displaystyle ^{H}D^{p_1 , q_1} ( ^{H}D^{p_2 , q_2}+\lambda_2  )y(t)= g(t,x(t),y(t)), \quad t\in [a, b],\\
		x(a)=0 \quad , \quad  x(b)=\displaystyle \sum\limits_{i=1}^{n} \mu_i (I^{\nu_i}(y))  (\eta_i),\\
		y(a)=0 \quad , \quad  y(b)=\displaystyle \sum\limits_{j=1}^{m} \omega_j (I^{\sigma_j}(x))  (\xi_j).
	\end{cases}
\end{equation}
\quad Where $ ^{H}D^{\alpha_i, \beta_i},^{H}D^{p_i, q_i} , i=1,2 $ are the Hilfer fractional derivative of order $\alpha_i, p_i$, such that $0<\alpha_i, p_i<1$ and parameter $\beta_i, q_i$, such that $0<\beta_i, q_i<1$, $i=1,2$, $\lambda_1, \lambda_2 \in \mathbb{R}$, $a\geq 0$, $I^{\nu_i}, I^{\sigma_j}$ are the Riemann-Liouville fractional integral of order $\nu_i$ and $\sigma_j$ respectively $\nu_i, \sigma_j > 0$, the points $\eta_i$, $\xi_j$ $\in [a, b]$, $i=1,2,....,n$, $j=1,2,....,m$, $\mu_i, \omega_j \in \mathbb{R}$ and $f,g:[a, b]\times \mathbb{R}\times \mathbb{R} \longrightarrow \mathbb{R}$ are continuous functions.\\
 
 This work is organized as follows : In section $2$, we recall some basic concepts of fractional calculus. In section $3$, we prove the first existence result using Laray-Schauder alternative, and by using Banach’s contraction mapping principle we prove the existence and uniqueness for our system $(2)$ as second result. In section $4$ we discuss the Ulam–Hyers stability and the generalized Ulam-Hyers stability of solutions for Hilfer coupled system.
\section{Preliminaries}
\quad Let us recall some basic definitions and notations of fractional calculus  which are needed throughout this paper.\\
 Let $\mathcal{K}=C([a, b], \mathbb{R})$ be the space equipped with the norm defined by \\
$\Vert x \Vert=sup \lbrace\vert x \vert, t\in [a, b]\rbrace $, $\big ( \mathcal{K}, \Vert . \Vert \big)$ is a Banach space,  and the product space $\big ( \mathcal{K}\times \mathcal{K}, \Vert . \Vert \big)$ is a Banach space equipped with the norm $\Vert (x, y) \Vert_{\mathcal{K}\times \mathcal{K}} = \Vert x \Vert + \Vert y \Vert$ for $(x, y) \in \mathcal{K}\times \mathcal{K}$.

\begin{definition}\cite{ref13}
 The Riemann-Liouville fractional integral of order $\alpha > 0$ for a continuous function  $f: [a, \infty) \longrightarrow \mathbb{R}$ can be defined as
\begin{equation}
I^\alpha f(t)=\displaystyle \frac{1}{\Gamma(\alpha )} \int_a^t (t-s)^{\alpha-1}f(s)ds.
\end{equation}
\quad Provided that the right-hand side exists on $(a, \infty)$. Where $\Gamma(z)$ is the Gamma function defined by $\Gamma(z)=\displaystyle\int_0^\infty e^{-t}t^{z-1}dt$, $\mathfrak{R}(z)>0$.
\end{definition}
\begin{definition}\cite{ref13}
The Riemann-Liouville fractional derivative of order $\alpha > 0$ of a continuous function $f$ is defined by
  \begin{equation}
^{RL} D^\alpha f(t):= D^{n}I^{n-\alpha}f(t)=\displaystyle \frac{1}{\Gamma(n-\alpha )} \Big ( \frac{d}{dt}\Big)^{n }\int_a^t (t-s)^{n-\alpha-1}f(s)ds.
\end{equation}
Where $n-1 < \alpha < n, n=[\alpha]+1,$ and $[\alpha]$ denotes the integer part of the real number $\alpha$.
 \end{definition}
 \begin{definition}\cite{ref13}
The Caputo  fractional derivative  of order $\alpha > 0$ of a continuous function $f$ is defined by
\begin{equation}
^C D^\alpha f(t):=I^{n -\alpha}D^{n}f(t)=\displaystyle \frac{1}{\Gamma(n-\alpha )} \int_0^t (t-s)^{n-\alpha-1}\Big ( \frac{d}{dt}\Big)^{n }f(s)ds.
\end{equation}
Where $n-1 < \alpha < n, n=[\alpha]+1,$ and $[\alpha]$ denotes the integer part of the real number $\alpha$.
\end{definition}
\begin{definition}(Hilfer fractional derivative \cite{ref11}) 
The Hilfer fractional derivative of order $\alpha$ and parameter $\beta$ of a function (also known as the generalized Riemann-Liouville fractional derivative) is defined by 
\begin{equation}
^H D^{\alpha,\beta} f(t)=I^{\beta (n-\alpha)}D^{n}I^{(1-\beta)(n-\alpha)}f(t).
\end{equation}
Where $n-1 < \alpha <n$, $0\leq \beta \leq 1$, $t > a$, $D=\displaystyle \Big ( \frac{d}{dt}\Big)$.
\end{definition}
\textbf{Remark :} When $\beta =0$, the Hilfer fractional derivative corresponds to the Riemann-Liouville fractional derivative:
\begin{equation}
^H D^{\alpha,0} f(t)=D^{n}I^{(n-\alpha)}f(t).
\end{equation}
While $\beta =1$, The Hilfer fractional derivative corresponds to the Caputo fractional derivative: 
\begin{equation}
^H D^{\alpha,1} f(t)=I^{ (n-\alpha)}D^{n}f(t).
\end{equation}
 The following lemma plays a fundamental role in establishing the existence results for the given problem.
\begin{lemma}\cite{ref11}
\quad Let $f \in L(a, b)$, $n-1< \alpha \leq n$, $n \in \mathbb{N}$, $0 \leq \beta \leq 1$, and $I^{(1-\beta)(n-\alpha)}f \in AC^{k}[a, b]$ then 
\begin{equation}
\Big ( I^{\alpha H}  D^{\alpha,\beta}f \Big ) (t)= f(t)- \displaystyle \sum\limits_{k=1}^{n-1} \frac{(t-a)^{k-(n-\alpha)(1-\beta)}}{\Gamma(k-(n-\alpha)(1-\beta)+1)}.\lim_{t \to +a}\frac{d^{k}}{dt^{k}}\Big ( I^{(1-\beta)(n-\alpha)}f \Big)(t).
\end{equation}
\end{lemma}
\quad The following lemma deals with a linear variant of the boundary value problem $(1)$.
\begin{lemma}
\quad Let $a \geq 0$, $0< \alpha _1, \alpha_2, p_1, p_2  < 1$, $0< \beta_1, \beta_2, q_1, q_2< 0$, and $\gamma _i= \alpha _i + \beta _i -\alpha _i \beta _i$, $\delta_i=p_i+q_i-p_iq_i$, with $i=1,2$, $1<\alpha_1 + \alpha _2 \leq 2$, $1<p_1+p_2 \leq 2$ and $h_1, h_2 \in C([a, b], \mathbb{R})$. Then the solution for the linear system of Hilfer fractional Langevin differential equations of the form:
\begin{equation}
	\begin{cases}
	\displaystyle ^{H}D^{\alpha_1 , \beta_1} ( ^{H}D^{\alpha_2 , \beta_2}+\lambda_1  )x(t)= h_1(t), \quad t\in [a, b],\\
	\displaystyle ^{H}D^{p_1 , q_1} ( ^{H}D^{p_2 , q_2}+\lambda_2  )y(t)= h_2(t), \quad t\in [a, b],\\
		x(a)=0 \quad , \quad  x(b)=\displaystyle \sum\limits_{i=1}^{n} \mu_i (I^{\nu_i}(y))  (\eta_i),\\
		y(a)=0 \quad , \quad  y(b)=\displaystyle \sum\limits_{j=1}^{m} \omega_j (I^{\sigma_j}(x))  (\xi_j),
	\end{cases}
\end{equation}
is equivalent to the integral equations
\begin{align}
x(t)&=I^{\alpha_1 + \alpha_2}h_1(t)- \lambda_1 I^{\alpha_2}x(t)+\displaystyle \frac{(t-a)^{\gamma_1+\alpha_2-1}}{\Lambda \Gamma(\gamma_1+\alpha_2)}\Bigg [ \Phi_4 \Bigg ( -I^{\alpha_1 + \alpha_2}h_1(b)\nonumber\\
&-\displaystyle \sum\limits_{i=1}^{n} \mu_i I^{p_1 + p_2+\nu_i}h_2(\eta_i) +\lambda_1 I^{\alpha_2}x(b)-\lambda_2 \displaystyle \sum\limits_{i=1}^{n} \mu_i I^{p_2+\nu_i}y(\eta_i)\Bigg ) \nonumber \\
&\displaystyle + \Phi_2 \Bigg ( \displaystyle \sum\limits_{j=1}^{m} \omega_j I^{\alpha_1 + \alpha_2+\sigma_j}h_1(\xi_j)-I^{p_1 + p_2}h_2(b)-\displaystyle \sum\limits_{j=1}^{m} \omega_j I^{\alpha_2+\sigma_j}x(\xi_j)\nonumber\\
&+ \lambda_2 I^{p_2}y(b)\Bigg) \Bigg],
\end{align}
and 
\begin{align}
y(t)&=I^{p_1 + p_2}h_2(t)- \lambda_2 I^{p_2}y(t)+\displaystyle \frac{(t-a)^{\delta_1+p_2-1}}{\Lambda \Gamma(\delta_1+p_2)}\Bigg [ \Phi_1 \Bigg (\displaystyle \sum\limits_{j=1}^{m} \omega_j I^{\alpha_1 + \alpha_2+\sigma_j}h_1(\xi_j)\nonumber\\
&-I^{p_1 +p_2}h_2(b)- \lambda_1 \displaystyle \sum\limits_{j=1}^{m} \omega_j I^{\alpha_2+\sigma_j}x(\xi_j) +\lambda_2 I^{p_2}y(b) \Bigg ) +\displaystyle  \Phi_3 \Bigg ( - I^{\alpha_1 + \alpha_2}h_1(b)\nonumber \\
&+ \displaystyle \sum\limits_{i=1}^{n} \mu_i I^{p_1+p_2+\nu_i}h_2(\eta_i)+\lambda_1 I^{\alpha_2}x(b)-\lambda_2 \displaystyle \sum\limits_{i=1}^{n} \mu_i I^{p_2+\nu_i}y(\eta_i) \Bigg) \Bigg],
\end{align}
where 
\begin{equation}
\Phi_1=\displaystyle\frac{(b-a)^{\gamma_1+\alpha_2-1}}{\Gamma(\gamma_1+\alpha_2)},
\end{equation}
\begin{equation}
\Phi_2 = \displaystyle \sum\limits_{i=1}^{n} \mu_i \frac{(\eta_i -a)^{\delta_1+p_2+\nu_i-1}}{\Gamma(\delta_1+p_2+\nu_i)},
\end{equation}
\begin{equation}
\Phi_3= \displaystyle \sum\limits_{j=1}^{m} \omega_j \frac{(\xi_j -a)^{\gamma_1+\alpha_2+\sigma_j-1}}{\Gamma(\gamma_1+\alpha_2+\sigma_j)},
\end{equation}
\begin{equation}
\Phi_4=\displaystyle\frac{(b-a)^{\delta_1+p_2-1}}{\Gamma(\delta_1+p_2)},
\end{equation}
\begin{equation}
\Lambda=\Phi_1\Phi_4-\Phi_2\Phi_3 \neq 0.
\end{equation}
\end{lemma}
\begin{proof} 
Applying the Riemann-Liouville fractional integral of order $\alpha_1$  to both sides of $(10)$ we obtain by using Lemma $2.5$
\begin{equation}
^{H}D^{\alpha_2 , \beta_2}x(t)+\lambda x(t)=I^{\alpha_1}h_1(t)+\frac{c_0}{\Gamma(\gamma_1)}(t-a)^{\gamma_1 -1},
\end{equation}
where $c_0$ constant and $\gamma_1=\alpha_1 +\beta_1 - \alpha_1 \beta_1$. Applying the Riemann-Liouville fractional integral of order $\alpha_2$  to both sides of $(18)$ we obtain by using Lemma $2.5$
\begin{equation}
x(t)=I^{\alpha_1 +\alpha_2}h_1(t)- \lambda_1 I^{\alpha_2}x(t)+\frac{c_0}{\Gamma(\gamma_1+\alpha_{2})}(t-a)^{\gamma_1+\alpha_2 -1}+\frac{c_1}{\Gamma(\gamma_2)}(t-a)^{\gamma_2 -1},
\end{equation}
from using the boundary condition $x(a)=0$ in $(19)$ we obtain that $c_1=0$. Then, we get
\begin{equation}
x(t)=I^{\alpha_1 +\alpha_2}h_1(t)- \lambda_1 I^{\alpha_2}x(t)+\frac{c_0}{\Gamma(\gamma_1+\alpha_{2})}(t-a)^{\gamma_1+\alpha_2 -1}.
\end{equation}
\quad In the same process, for $y$ we get by using the Riemann-Liouville fractional integral of order $p_1$ and $p_2$, respectively, on both side of the second equation in $(10)$, and by using Lemma $2.5$ we get 
\begin{equation}
y(t)=I^{p_1 +p_2}h_2(t)- \lambda_2 I^{p_2}y(t)+\frac{d_0}{\Gamma(\delta_1+p_{2})}(t-a)^{\delta_1+p_2 -1}+\frac{d_1}{\Gamma(\delta_2)}(t-a)^{\delta_2 -1},
\end{equation}
 where $d_0, d_1$ are constants, and by using the boundary condition $y(a)=0$ in $(21)$ we obtain that $d_1=0$. Then we get 
 \begin{equation}
y(t)=I^{p_1 +p_2}h_2(t)- \lambda_2 I^{p_2}y(t)+\frac{d_0}{\Gamma(\delta_1+p_{2})}(t-a)^{\delta_1+p_2 -1},
\end{equation}
from using the boundary conditions $x(b)=\sum\limits_{i=1}^{n} \mu_i (I^{\nu_i}(y))(\eta_i)$ and\\
 $y(b)=\sum\limits_{j=1}^{m}\omega_j (I^{\sigma_j}(x))(\xi_j)$ in $(20)$ and $(22)$, we get the system
 \begin{equation}
  \begin{cases}
 \Phi_1c_0-\Phi_2d_0=\Omega_1,\\
 -\Phi_2c_0+\Phi_4d_0=\Omega_2,
 \end{cases}
  \end{equation}
  where 
  \begin{equation}
  \Omega_1= -I^{\alpha_1 + \alpha_2}h_1(b)-\displaystyle \sum\limits_{i=1}^{n} \mu_i I^{p_1 + p_2+\nu_i}h_2(\eta_i) +\lambda_1 I^{\alpha_2}x(b)-\lambda_2 \displaystyle \sum\limits_{i=1}^{n} \mu_i I^{p_2+\nu_i}y(\eta_i),
  \end{equation}
  \begin{equation}
  \Omega_2=\sum\limits_{j=1}^{m} \omega_j I^{\alpha_1 + \alpha_2+\sigma_j}h_1(\xi_j)-I^{p_1 + p_2}h_2(b)-\displaystyle \sum\limits_{j=1}^{m} \omega_j I^{\alpha_2+\sigma_j}x(\xi_j)+ \lambda_2 I^{p_2}y(b).
  \end{equation}
  Solving the system $(23)$ we obtain 
  \begin{equation}
  c_0=\frac{\Phi_4\Omega_1+\Phi_2\Omega_2}{\Phi_1\Phi_4-\Phi_2\Phi_3} \quad , \quad d_0=\frac{\Phi_1\Omega_2+\Phi_3\Omega_1}{\Phi_1\Phi_4-\Phi_2\Phi_3}.
  \end{equation}
  Substituting the value of $c_0$, and $d_0$ in $(20)$ and $(22)$, respectively yields the solution $(11)$ and $(12)$. The converse follows by direct computation. This completes the proof.
  \end{proof}
 \begin{theorem}(Leray–Schauder alternative \cite{ref9})
\quad Let $X$ be a Banach space, $C$ a closed subset of $X$, $\mathcal{A}: C \longrightarrow C$ be a completely continuous operator. Let 
\begin{equation*}
\mathfrak{M}(\mathcal{A})=\lbrace x \in C : x=\theta \mathcal{A}x \quad ; \quad 0<\theta<1\rbrace.
\end{equation*}
Then either the set $\mathfrak{M}(\mathcal{A})$ is unbounded, or $\mathcal{A}$ has at least one fixed point.
\end{theorem}
\begin{theorem}(Banach fixed point theorem )
Let $X$ be a Banach space, $C$ a closed subset of $X$, and $\mathcal{A}: C \longrightarrow C$ be a strict contraction. Then $\mathcal{A}$ has a unique fixed point in $C$
\end{theorem}

\section{Main Results }
\quad In view of Lemma $2.6$, we define the operator $\mathcal{A}: \mathcal{K}\times \mathcal{K} \longrightarrow \mathcal{K}\times \mathcal{K} $ by
\begin{equation}
\mathcal{A}(x,y)(t)= \begin{pmatrix}
\mathcal{A}_1(x,y)(t)\\
\mathcal{A}_2(x,y)(t)
\end{pmatrix},
\end{equation}
where 
\begin{align}
\mathcal{A}_1&(x,y)(t)=I^{\alpha_1 + \alpha_2}f(t,x(t),y(t))- \lambda_1 I^{\alpha_2}x(t)+\displaystyle \frac{(t-a)^{\gamma_1+\alpha_2-1}}{\Lambda \Gamma(\gamma_1+\alpha_2)}\nonumber \\
&\times\Bigg [ \Phi_4 \Bigg ( -I^{\alpha_1 + \alpha_2}f(b,x(b),y(b)))-\displaystyle \sum\limits_{i=1}^{n} \mu_i I^{p_1 + p_2+\nu_i}g(\eta_i,x(\eta_i),y(\eta_i))\nonumber\\
& +\lambda_1 I^{\alpha_2}x(b)-\lambda_2 \displaystyle \sum\limits_{i=1}^{n} \mu_i I^{p_2+\nu_i}y(\eta_i)\Bigg )+ \Phi_2 \Bigg ( \displaystyle \sum\limits_{j=1}^{m} \omega_j I^{\alpha_1 + \alpha_2+\sigma_j}f(\xi_j,x(\xi_j),y(\xi_j)) \nonumber \\
& -I^{p_1 + p_2}g(b,x(b),y(b))-\displaystyle \sum\limits_{j=1}^{m} \omega_j I^{\alpha_2+\sigma_j}x(\xi_j)+ \lambda_2 I^{p_2}y(b)\Bigg) \Bigg],
\end{align}
and
\begin{align}
\mathcal{A}_2&(x,y)(t)=I^{p_1 + p_2}g(t,x(t),y(t))- \lambda_2 I^{p_2}y(t)+\displaystyle \frac{(t-a)^{\delta_1+p_2-1}}{\Lambda \Gamma(\delta_1+p_2)}\nonumber\\
&\times \Bigg [ \Phi_1 \Bigg (\displaystyle \sum\limits_{j=1}^{m} \omega_j I^{\alpha_1 + \alpha_2+\sigma_j}f(\xi_j,x(\xi_j),y(\xi_j))-I^{p_1 +p_2}g(b,x(b),y(b))\nonumber\\
&- \lambda_1 \displaystyle \sum\limits_{j=1}^{m} \omega_j I^{\alpha_2+\sigma_j}x(\xi_j) +\lambda_2 I^{p_2}y(b) \Bigg ) +\displaystyle \Phi_3 \Bigg ( - I^{\alpha_1 + \alpha_2}f(b,x(b),y(b))\nonumber \\
&+ \displaystyle \sum\limits_{i=1}^{n} \mu_i I^{p_1+p_2+\nu_i}g(\eta_i,x(\eta_i),y(\eta_i))+\lambda_1 I^{\alpha_2}x(b)-\lambda_2 \displaystyle \sum\limits_{i=1}^{n} \mu_i I^{p_2+\nu_i}y(\eta_i) \Bigg)\Bigg].
\end{align}
To simplify the computations, we use the following notations:
\begin{equation}
X_1=\vert \lambda_1 \vert \Bigg\lbrace \frac{(b-a)^{\alpha_2}}{\Gamma (\alpha_2+1)}+\frac{(b-a)^{\gamma_1+\alpha_2-1}}{\vert \Lambda \vert \Gamma(\gamma_1+\alpha_2)}\Bigg [ \vert \Phi_4 \vert \frac{(b-a)^{\alpha_2}}{\Gamma (\alpha_2+1)} +\vert \Phi_2 \vert \displaystyle \sum\limits_{j=1}^{m} \vert \omega_j \vert \frac{(\xi_j-a)^{\alpha_2+\sigma_j}}{\Gamma(\alpha_2+\sigma_j+1)}\Bigg] \Bigg \rbrace,
\end{equation}
\begin{equation}
Y_1=\vert \lambda_2 \vert \Bigg\lbrace \frac{(b-a)^{\gamma_1+\alpha_2-1}}{\vert \Lambda \vert \Gamma(\gamma_1+\alpha_2)}\Bigg [ \vert \Phi_4 \vert  \sum\limits_{i=1}^{n} \vert \mu_i \vert \frac{(\eta_i-a)^{p_2+\nu_i}}{\Gamma(p_2+\nu_i+1)}+\vert \Phi_2\vert \frac{(b-a)^{p_2}}{\Gamma(p_2+1)}\Bigg] \Bigg \rbrace,
\end{equation}
\begin{align}
F_1=&\frac{(b-a)^{\alpha_1+\alpha_2}}{\Gamma(\alpha_1+\alpha_2+1)}\Bigg ( 1+\frac{(b-a)^{\gamma_1+\alpha_2-1}}{\vert \Lambda \vert \Gamma(\gamma_1+\alpha_2)}\vert \Phi_4\vert \Bigg )\nonumber\\
&+\frac{(b-a)^{\gamma_1+\alpha_2-1}}{\vert \Lambda \vert \Gamma(\gamma_1+\alpha_2)}\vert \Phi_2\vert \displaystyle \sum\limits_{j=1}^{m} \vert \omega_j \vert \frac{(\xi_j-a)^{\alpha_1+\alpha_2+\sigma_j}}{\Gamma(\alpha_1+\alpha_2+\sigma_j+1)}, 
\end{align}
\begin{equation}
G_1=\frac{(b-a)^{\gamma_1+\alpha_2-1}}{\vert \Lambda \vert \Gamma(\gamma_1+\alpha_2)}\Bigg( \vert \Phi_4 \vert \sum\limits_{i=1}^{n} \vert \mu_i \vert \frac{(\eta_i-a)^{p_1+p_2+\nu_i}}{\Gamma(p_1+p_2+\nu_i+1)} +\vert \Phi_2\vert \frac{(b-a)^{p_1+p_2}}{\Gamma(p_1+p_2+1)}  \Bigg),
\end{equation}
\begin{equation}
X_2=\vert \lambda_1 \vert \Bigg\lbrace \frac{(b-a)^{\delta_1+p_2-1}}{\vert \Lambda \vert \Gamma(\delta_1+p_2)}\Bigg [ \vert \Phi_1 \vert  \sum\limits_{j=1}^{m} \vert \omega_j \vert \frac{(\omega_j-a)^{\alpha_2+\sigma_j}}{\Gamma(\alpha_2+\sigma_j+1)}+\vert \Phi_3\vert \frac{(b-a)^{\alpha_2}}{\Gamma(\alpha_2+1)}\Bigg] \Bigg \rbrace,
\end{equation}
\begin{equation}
Y_2=\vert \lambda_2 \vert \Bigg\lbrace \frac{(b-a)^{p_2}}{\Gamma (p_2+1)}+\frac{(b-a)^{\delta_1+p_2-1}}{\vert \Lambda \vert \Gamma(\delta_1+p_2)}\Bigg [ \vert \Phi_1\vert \frac{(b-a)^{p_2}}{\Gamma (p_2+1)} +\vert \Phi_3 \vert \displaystyle \sum\limits_{i=1}^{n} \vert \mu_i \vert \frac{(\eta_i-a)^{p_2+\nu_i}}{\Gamma(p_2+\nu_i+1)}\Bigg] \Bigg \rbrace,
\end{equation}
\begin{equation}
F_2=\frac{(b-a)^{\delta_1+p_2-1}}{\vert \Lambda \vert \Gamma(\delta_1+p_2)}\Bigg( \vert \Phi_1 \vert  \sum\limits_{j=1}^{m} \vert \omega_j \vert \frac{(\omega_j-a)^{\alpha_2+\sigma_j}}{\Gamma(\alpha_2+\sigma_j+1)} + \vert \Phi_3\vert \frac{(b-a)^{\alpha_1+\alpha_2}}{\Gamma(\alpha_1+\alpha_2+1)}\Bigg ),
\end{equation}
\begin{align}
G_2=&\frac{(b-a)^{p_1+p_2}}{\Gamma(p_1+p_2+1)}\Bigg ( 1+\frac{(b-a)^{\delta_1+p_2-1}}{\vert \Lambda \vert \Gamma(\delta_1+p_2)}\vert \Phi_1\vert \Bigg )\nonumber\\
&+\frac{(b-a)^{\delta_1+p_2-1}}{\vert \Lambda \vert \Gamma(\delta_1+p_2)}\vert \Phi_3\vert \displaystyle \sum\limits_{i=1}^{n} \vert \mu_i \vert \frac{(\eta_i-a)^{p_1+p_2+\nu_i}}{\Gamma(p_1+p_2+\nu_i+1)}, 
\end{align}
\subsection{Existence result via Leray-Schauder alternative}
Our first result is based on Leray-Schauder alternative
\begin{theorem}
 Assume that $f, g: [a,b]\times \mathbb{R}\times\mathbb{R}\longrightarrow \mathbb{R}$, are
continuous functions, and there exist real constants $M_i, \overline{M}_i \geq 0$, $i=1,2,3$, such that, for $x, y \in \mathbb{R}$,
\begin{equation}
\vert f(t,x(t),y(t)) \vert\leq M_1+M_2\vert x\vert +M_3\vert y \vert,
\end{equation} 
\begin{equation}
\vert g(t,x(t),y(t)) \vert\leq \overline{M}_1+\overline{M}_2\vert x\vert +\overline{M}_3\vert y \vert,
\end{equation} 
if 
\begin{equation}
\mathcal{K}_1=(F_1+F_2)M_2+(G_1+G_2)\overline{M}_2+(X_1+X_2)<1,
\end{equation}
\begin{equation}
\mathcal{K}_2=(F_1+F_2)M_3+(G_1+G_2)\overline{M}_3+(Y_1+Y_2)<1,
\end{equation}
then, the system $(2)$ has at least one solution on $[a, b]$. Where $X_i, Y_i, F_i, G_i$, $i=1, 2$ are given by $(30)-(37)$.
\end{theorem}
\begin{proof}
The operator $\mathcal{A}$ defined in $(27)$ is continuous, by the continuity of functions $f$ and $g$. We will show that the operator $\mathcal{A} : \mathcal{K}\times \mathcal{K}\longrightarrow \mathcal{K}\times \mathcal{K}$ is completely continuous. Let $\mathcal{B}_r=\lbrace (x, y)\in \mathcal{K}\times \mathcal{K} : \Vert (x,y)\Vert\leq r \rbrace$ be bounded set in $ \mathcal{K}\times \mathcal{K}$. Then, for any $(x, y)\in \mathcal{B}_r $, there exist positive real numbers $\overline{f}$ and $\overline{g}$ such that $\vert f(t,x(t),y(t)\vert\leq \overline{f}$ and $\vert g(t,x(t),y(t)\vert\leq \overline{g}$. Then, for any $(x, y)\in \mathcal{B}_r$
 
\begin{align*}
\big \Vert \mathcal{A}_1(x,y) \big \Vert&\leq \underset{t \in [a, b ]}{sup} \Bigg \lbrace I^{\alpha_1 + \alpha_2}\vert f(t, x(t), y(t))\vert+ \vert \lambda_1 \vert I^{\alpha_2}\vert x(t)\vert +\displaystyle \frac{(t-a)^{\gamma_1+\alpha_2-1}}{\vert \Lambda \vert \Gamma(\gamma_1+\alpha_2)}\\
&\times\Bigg [\vert \Phi_4\vert \Bigg (  I^{\alpha_1 + \alpha_2}\vert f(b, x(b), y(b))\vert+ \displaystyle \sum\limits_{i=1}^{n} \vert \mu_i\vert I^{p_1 + q_2+\nu_i}\vert g(\eta_i, x(\eta_i), y(\eta_i))\vert \\
&+\vert \lambda_1 \vert I^{\alpha_2}\vert x(b)\vert+\vert \lambda_2 \vert \displaystyle \sum\limits_{i=1}^{n} \vert \mu_i\vert  I^{p_2+\nu_i}\vert y(\eta_i)\vert \Bigg )\\
& +\vert \Phi_2\vert \Bigg ( \displaystyle \sum\limits_{j=1}^{m} \vert \omega_j\vert I^{\alpha_1 + \alpha_2+\sigma_j}\vert f(\xi_j, x(\xi_j), y(\xi_j))\vert+I^{p_1 + p_2}\vert g(b, x(b), y(b))\vert \\
&+\displaystyle \sum\limits_{j=1}^{m} \vert \omega_j\vert I^{\alpha_2+\sigma_j}\vert x(\xi_j)\vert + \vert \lambda_2 \vert I^{p_2}\vert y(b)\vert \Bigg) \Bigg ]\Bigg \rbrace \\
&\leq\displaystyle\frac{(b-a)^{\alpha_1 +\alpha_2}}{\Gamma (\alpha_1 +\alpha_2+1)}\overline{f} + \vert \lambda_1 \vert \frac{(b-a)^{\alpha_2}}{\Gamma (\alpha_2+1)}\Vert x \Vert+ \frac{(b-a)^{\gamma_1+\alpha_2-1}}{\vert \Lambda \vert \Gamma(\gamma_1+\alpha_2)} \\
&\times \Bigg [ \vert \Phi_4\vert \Bigg (\frac{(b-a)^{\alpha_1 +\alpha_2}}{\Gamma (\alpha_1 +\alpha_2+1)}\overline{f} + \displaystyle \sum\limits_{i=1}^{n} \vert \mu_i \vert \frac{(\eta_i-a)^{p_1 +p_2+\nu_i}}{\Gamma(p_1 +p_2+\nu_i+1)}\overline{g}\\
& + \vert \lambda_1 \vert  \frac{(b-a)^{\alpha_2}}{\Gamma (\alpha_2+1)}\Vert x \Vert+\vert \lambda_2\vert \displaystyle \sum\limits_{i=1}^{n} \vert \mu_i \vert \frac{(\eta_i-a)^{p_2+\nu_i}}{\Gamma(p_2+\nu_i+1)}\Vert y \Vert \Bigg)\\
&+ \vert\Phi_2 \vert \Bigg ( \displaystyle \sum\limits_{j=1}^{m} \vert \omega_j\vert\frac{(\xi_j-a)^{\alpha_1+\alpha_2+\sigma_j}}{\Gamma (\alpha_1+\alpha_2+\sigma_j+1)}\overline{f}+\frac{(b-a)^{p_1 +p_2}}{\Gamma (p_1 +p_2+1)}\overline{g}\\
&+\vert \lambda_1\vert \displaystyle \sum\limits_{j=1}^{m} \vert \omega_j\vert\frac{(\xi_j-a)^{\alpha_2+\sigma_j}}{\Gamma (\alpha_2+\sigma_j+1)} \Vert x \Vert + \vert \lambda_2\vert \frac{(b-a)^{p_2}}{p_2+1}\Vert y \Vert \Bigg) \Bigg]  \\
&\leq F_1\overline{f}+X_1 \Vert x \Vert+G_1\overline{g} +Y_1 \Vert y\Vert,
\end{align*}
which implies that 
\begin{equation}
\Vert \mathcal{A}_1 (x,y) \Vert \leq F_1\overline{f}+X_1 \Vert x \Vert+G_1\overline{g} +Y_1 \Vert y\Vert.
\end{equation}
Similarly, we have that
\begin{equation}
\Vert \mathcal{A}_2 (x,y) \Vert \leq F_2\overline{f}+X_2 \Vert x \Vert+G_2\overline{g} +Y_2 \Vert y\Vert.
\end{equation}
From $(42),(43)$, it follows that 
\begin{equation}
\Vert \mathcal{A} (x,y) \Vert \leq [F_1+F_2]\overline{f}+[G_1+G_2]\overline{g}+[X_1+X_2]r+ [Y_1+Y_2]r.
\end{equation}
\quad Then the operator $\mathcal{A}$ is uniformly bounded.\\
Next, we show that the operator $\mathcal{A}$ is completely continuous. Let $ t_1, t _2 \in [a, b]; t_1 < t_2$ then for any $(x, y)\in \mathcal{B}_r$we have 

 \begin{align*}
\big \vert &\mathcal{A}_1(x(t_2),y(t_2))-\mathcal{A}_1(x(t_1),y(t_1)) \Big \vert \leq \frac{1}{\Gamma(\alpha_1+\alpha_2)} \Big \vert \displaystyle\int_a^{t_1}\Big ( (t_2-s)^{\alpha_1+\alpha_2-1}\\
&-(t_1-s)^{\alpha_1+\alpha_2-1}\Big )f(s,x(s),y(s))ds + \displaystyle\int_{t_1}^{t_2} (t_2-s)^{\alpha_1+\alpha_2-1} f(s,x(s),y(s))ds \Big \vert\\
&+\frac{\vert \lambda_1 \vert}{\Gamma(\alpha_2)}\Big \vert \displaystyle\int_a^{t_1}\Big ( (t_2-s)^{\alpha_2-1}-(t_1-s)^{\alpha_2-1}\Big )x(s)ds\\
&+ \displaystyle\int_{t_1}^{t_2} (t_2-s)^{\alpha_2-1} x(s)ds \Big \vert\displaystyle + \frac{\vert (t_2-a)^{\gamma_1+\alpha_2-1}-(t_1-a)^{\gamma_1+\alpha_2-1}\vert}{\vert\Lambda \vert \Gamma(\gamma_1+\alpha_2)}\\
&\times\Bigg [\vert \Phi_4\vert \Bigg (  I^{\alpha_1 + \alpha_2}\vert f(b, x(b), y(b))\vert+ \displaystyle \sum\limits_{i=1}^{n} \vert \mu_i\vert I^{p_1 + q_2+\nu_i}\vert g(\eta_i, x(\eta_i), y(\eta_i))\vert \\
&+\vert \lambda_1 \vert I^{\alpha_2}\vert x(b)\vert+\vert \lambda_2 \vert \displaystyle \sum\limits_{i=1}^{n} \vert \mu_i\vert  I^{p_2+\nu_i}\vert y(\eta_i)\vert \Bigg )\\
& +\vert \Phi_2\vert \Bigg ( \displaystyle \sum\limits_{j=1}^{m} \vert \omega_j\vert I^{\alpha_1 + \alpha_2+\sigma_j}\vert f(\xi_j, x(\xi_j), y(\xi_j))\vert+I^{p_1 + p_2}\vert g(b, x(b), y(b))\vert \\
&+\displaystyle \sum\limits_{j=1}^{m} \vert \omega_j\vert I^{\alpha_2+\sigma_j}\vert x(\xi_j)\vert + \vert \lambda_2 \vert I^{p_2}\vert y(b)\vert \Bigg) \Bigg ] \\
& \leq \frac{\overline{f}}{\Gamma(\alpha_1+\alpha_2)} \Bigg( \Big \vert\displaystyle\int_a^{t_1}\Big ( (t_2-s)^{\alpha_1+\alpha_2-1}-(t_1-s)^{\alpha_1+\alpha_2-1}\Big )ds \Big \vert + \Big \vert\displaystyle\int_{t_1}^{t_2} (t_2-s)^{\alpha_1+\alpha_2-1}ds \Big \vert \Bigg)  \\
&+\frac{\vert \lambda_1 \vert}{\Gamma(\alpha_2)}\Big \vert \displaystyle\int_a^{t_1}\Big ( (t_2-s)^{\alpha_2-1}-(t_1-s)^{\alpha_2-1}\Big )x(s)ds+ \displaystyle\int_{t_1}^{t_2} (t_2-s)^{\alpha_2-1} x(s)ds \Big \vert\\
&+ \frac{\big \vert(t_2-a)^{\gamma_1+\alpha_2-1}-(t_1-a)^{\gamma_1+\alpha_2-1}\big \vert}{\vert\Lambda \vert \Gamma(\gamma_1+\alpha_2)}\Bigg [ \vert \Phi_4\vert \Bigg (\frac{(b-a)^{\alpha_1 +\alpha_2}}{\Gamma (\alpha_1 +\alpha_2+1)}\overline{f}+\vert \lambda_1 \vert  \frac{(b-a)^{\alpha_2}}{\Gamma (\alpha_2+1)}\Vert x \Vert  \\
&+\displaystyle \sum\limits_{i=1}^{n} \vert \mu_i \vert \frac{(\eta_i-a)^{p_1 +p_2+\nu_i}}{\Gamma(p_1 +p_2+\nu_i+1)}\overline{g}+\vert \lambda_2\vert \displaystyle \sum\limits_{i=1}^{n} \vert \mu_i \vert \frac{(\eta_i-a)^{p_2+\nu_i}}{\Gamma(p_2+\nu_i+1)}\Vert y \Vert \Bigg)\\  
&+ \vert\Phi_2 \vert \Bigg ( \displaystyle \sum\limits_{j=1}^{m} \vert \omega_j\vert\frac{(\xi_j-a)^{\alpha_1+\alpha_2+\sigma_j}}{\Gamma (\alpha_1+\alpha_2+\sigma_j+1)}\overline{f}+\frac{(b-a)^{p_1 +p_2}}{\Gamma (p_1 +p_2+1)}\overline{g}\\
&+\vert \lambda_1\vert \displaystyle \sum\limits_{j=1}^{m} \vert \omega_j\vert\frac{(\xi_j-a)^{\alpha_2+\sigma_j}}{\Gamma (\alpha_2+\sigma_j+1)} \Vert x \Vert + \vert \lambda_2\vert \frac{(b-a)^{p_2}}{p_2+1}\Vert y \Vert \Bigg) \Bigg]\\
& \leq \frac{\overline{f}}{\Gamma(\alpha_1+\alpha_2)}\Bigg( \Big \vert\displaystyle\int_a^{t_1}\Big ( (t_2-s)^{\alpha_1+\alpha_2-1}-(t_1-s)^{\alpha_1+\alpha_2-1}\Big )ds \Big \vert + \Big \vert \displaystyle\int_{t_1}^{t_2} (t_2-s)^{\alpha_1+\alpha_2-1}ds \Big \vert\Bigg)  \\
&+\frac{r\vert \lambda_1 \vert}{\Gamma(\alpha_2)}\Bigg( \Big \vert \displaystyle\int_a^{t_1}\Big ( (t_2-s)^{\alpha_2-1}-(t_1-s)^{\alpha_2-1}\Big )ds \Big \vert+ \Big \vert\displaystyle\int_{t_1}^{t_2} (t_2-s)^{\alpha_2-1} ds \Big \vert \Bigg)\\
&+\frac{\big \vert(t_2-a)^{\gamma_1+\alpha_2-1}-(t_1-a)^{\gamma_1+\alpha_2-1}\big \vert}{\vert\Lambda \vert \Gamma(\gamma_1+\alpha_2)}\Bigg [ \vert \Phi_4\vert \Bigg (\frac{(b-a)^{\alpha_1 +\alpha_2}}{\Gamma (\alpha_1 +\alpha_2+1)}\overline{f}+\vert \lambda_1 \vert  \frac{r(b-a)^{\alpha_2}}{\Gamma (\alpha_2+1)} 
\end{align*}
\begin{align*}
+\displaystyle& \sum\limits_{i=1}^{n} \vert \mu_i \vert \frac{(\eta_i-a)^{p_1 +p_2+\nu_i}}{\Gamma(p_1 +p_2+\nu_i+1)}\overline{g}+\vert \lambda_2\vert r \displaystyle \sum\limits_{i=1}^{n} \vert \mu_i \vert \frac{(\eta_i-a)^{p_2+\nu_i}}{\Gamma(p_2+\nu_i+1)} \Bigg)\\  
&+ \vert\Phi_2 \vert \Bigg ( \displaystyle \sum\limits_{j=1}^{m} \vert \omega_j\vert\frac{(\xi_j-a)^{\alpha_1+\alpha_2+\sigma_j}}{\Gamma (\alpha_1+\alpha_2+\sigma_j+1)}\overline{f}+\frac{(b-a)^{p_1 +p_2}}{\Gamma (p_1 +p_2+1)}\overline{g}\\
&+\vert \lambda_1\vert \displaystyle \sum\limits_{j=1}^{m} \vert \omega_j\vert\frac{r(\xi_j-a)^{\alpha_2+\sigma_j}}{\Gamma (\alpha_2+\sigma_j+1)} + \vert \lambda_2\vert \frac{r(b-a)^{p_2}}{\Gamma(p_2+1)} \Bigg) \Bigg].
\end{align*}
In the same process, we can obtain

\begin{align*}
\big \vert &\mathcal{A}_2(x(t_2),y(t_2))-\mathcal{A}_2(x(t_1),y(t_1)) \Big \vert \leq \frac{\overline{g}}{\Gamma(p_1+p_2)}\Bigg( \Big \vert\displaystyle\int_a^{t_1}\Big ( (t_2-s)^{p_1+p_2-1}\\
&-(t_1-s)^{p_1+p_2-1}\Big )ds \Big \vert + \Big \vert \displaystyle\int_{t_1}^{t_2} (t_2-s)^{p_1+p_2-1}ds \Big \vert\Bigg)  \\
&+\frac{r\vert \lambda_2 \vert}{\Gamma(p_2)}\Bigg( \Big \vert \displaystyle\int_a^{t_1}\Big ( (t_2-s)^{p_2-1}-(t_1-s)^{p_2-1}\Big )ds \Big \vert+ \Big \vert\displaystyle\int_{t_1}^{t_2} (t_2-s)^{p_2-1} ds \Big \vert \Bigg)\\
&+\frac{\big \vert(t_2-a)^{\delta_1+p_2-1}-(t_1-a)^{\delta_1+p_2-1}\big \vert}{\vert\Lambda \vert \Gamma(\delta_1+p_2)}\Bigg [ \vert \Phi_1\vert \Bigg (\frac{(b-a)^{p_1 +p_2}}{\Gamma (p_1 +p_2+1)}\overline{g}+\vert \lambda_2 \vert  \frac{r(b-a)^{p_2}}{\Gamma (p_2+1)} \\
&+\displaystyle \sum\limits_{j=1}^{m} \vert \omega_j \vert \frac{(\xi_j-a)^{\alpha_1 +\alpha_2+\sigma_j}}{\Gamma(\alpha_1 +\alpha_2+\sigma_j+1)}\overline{f}+\vert \lambda_1\vert r \displaystyle \sum\limits_{j=1}^{m} \vert \omega_j \vert \frac{(\xi_j-a)^{\alpha_2+\sigma_j}}{\Gamma(\alpha_2+\sigma_j+1)} \Bigg)\\  
&+ \vert\Phi_3 \vert \Bigg ( \displaystyle \sum\limits_{i=1}^{n} \vert \mu_i\vert\frac{(\eta_i-a)^{p_1+p_2+\nu_i}}{\Gamma (p_1+p_2+\nu_i+1)}\overline{g}+\frac{(b-a)^{\alpha_1 +\alpha_2}}{\Gamma (\alpha_1 +\alpha_2+1)}\overline{f}\\
&+\vert \lambda_2\vert \displaystyle \sum\limits_{i=1}^{n} \vert \mu_i\vert\frac{r(\eta_i-a)^{p_2+\nu_i}}{\Gamma (p_2+\nu_i+1)} + \vert \lambda_1\vert \frac{r(b-a)^{\alpha_2}}{\Gamma(\alpha_2+1)} \Bigg) \Bigg].
\end{align*}

 As $t _2 \longrightarrow t _1$ the right-hand side of the two above inequality tends to zero, implies that $\mathcal{A}(x, y)$ is equicontinuous.  Therefore it follows by Arzelà-Ascoli theorem that $\mathcal{A}(x, y)$ is relatively compact then $\mathcal{A}(x, y)$ is completely continuous.\\

 Next, we will prove that te set $\mathfrak{M}(\mathcal{A})=\lbrace(x, y)\in \mathcal{K}\times\mathcal{K}\mid\mathcal{A} (x, y)=\theta \mathcal{A}(x, y) ; 0<\theta<1\rbrace$, is bounded. \\
 
 Let $(x, y)\in \mathfrak{M}$, with $(x, y)=\theta \mathcal{A}(x, y)$, for any $t\in [a, b]$, we have
  \begin{equation}
  \begin{cases}
x(t)=\theta \mathcal{A}_{1}(x, y)(t),\\
y(t)=\theta \mathcal{A}_{2}(x, y)(t).
 \end{cases}
  \end{equation}
 Then 
 
 \begin{align*}
\big \vert x(t) \big \vert&\leq\displaystyle\frac{(b-a)^{\alpha_1 +\alpha_2}}{\Gamma (\alpha_1 +\alpha_2+1)}\big ( M_1+M_2\vert x \vert +M_3 \vert y \vert \big) + \vert \lambda_1 \vert \frac{(b-a)^{\alpha_2}}{\Gamma (\alpha_2+1)}\Vert x \Vert \\
&+ \frac{(b-a)^{\gamma_1+\alpha_2-1}}{\vert \Lambda \vert \Gamma(\gamma_1+\alpha_2)}\times \Bigg [ \vert \Phi_4\vert \Bigg (\frac{(b-a)^{\alpha_1 +\alpha_2}}{\Gamma (\alpha_1 +\alpha_2+1)}\big ( M_1+M_2\vert x \vert +M_3 \vert y \vert \big) \\
&+ \displaystyle \sum\limits_{i=1}^{n} \vert \mu_i \vert \frac{(\eta_i-a)^{p_1 +p_2+\nu_i}}{\Gamma(p_1 +p_2+\nu_i+1)}\big ( \overline{M}_1+\overline{M}_2\vert x \vert +\overline{M}_3 \vert y \vert \big) + \vert \lambda_1 \vert  \frac{(b-a)^{\alpha_2}}{\Gamma (\alpha_2+1)}\Vert x \Vert\\
&+\vert \lambda_2\vert \displaystyle \sum\limits_{i=1}^{n} \vert \mu_i \vert \frac{(\eta_i-a)^{p_2+\nu_i}}{\Gamma(p_2+\nu_i+1)}\Vert y \Vert \Bigg)+ \vert\Phi_2 \vert \Bigg ( \displaystyle \sum\limits_{j=1}^{m} \vert \omega_j\vert\frac{(\xi_j-a)^{\alpha_1+\alpha_2+\sigma_j}}{\Gamma (\alpha_1+\alpha_2+\sigma_j+1)}\\
&\times\big ( M_1+M_2\vert x \vert +M_3 \vert y \vert \big)+\frac{(b-a)^{p_1 +p_2}}{\Gamma (p_1 +p_2+1)}\big ( \overline{M}_1+\overline{M}_2\vert x \vert +\overline{M}_3 \vert y \vert \big) \\
& +\vert \lambda_1\vert \displaystyle \sum\limits_{j=1}^{m} \vert \omega_j\vert\frac{(\xi_j-a)^{\alpha_2+\sigma_j}}{\Gamma (\alpha_2+\sigma_j+1)} \Vert x \Vert+ \vert \lambda_2\vert \frac{(b-a)^{p_2}}{p_2+1}\Vert y \Vert \Bigg) \Bigg]  \\
&\leq F_1\big ( M_1+M_2\vert x \vert +M_3 \vert y \vert \big)+G_1\big ( \overline{M}_1+\overline{M}_2\vert x \vert +\overline{M}_3 \vert y \vert \big) +X_1 \Vert x \Vert +Y_1 \Vert y\Vert.
\end{align*}
Then,
\begin{equation}
\Vert x \Vert \leq \big ( F_1M_1+G_1\overline{M_1}\big)+\big ( F_1M_2+G_1\overline{M}_2+X_1\big)\Vert x \Vert + \big ( F_1M_3+G_1\overline{M}_3+Y_1\big)\Vert y \Vert.
\end{equation}
 In the same process, we can obtain 
 \begin{equation}
\Vert y \Vert \leq \big ( F_2M_1+G_1\overline{M}_1\big)+\big ( F_2M_2+G_2\overline{M}_2+X_2\big)\Vert x \Vert + \big ( F_2M_3+G_2\overline{M}_3+Y_2\big)\Vert y \Vert,
\end{equation}
which imply that
 \begin{align*}
\Vert x \Vert + \Vert y \Vert  &\leq (F_1+F_2)M_1+ (G_1+G_2)\overline{M}_1+\bigg [ (F_1+F_2)M_2+(G_1+G_2)\overline{M}_2\\
 & +(X_1+X_2)\bigg ]\Vert x\Vert +\bigg [ (F_1+F_2)M_3+(G_1+G_2)\overline{M}_3+(Y_1+Y_2)\bigg ]\Vert y\Vert,
  \end{align*}
  thus, we obtain
  \begin{equation}
 \Vert (x, y) \Vert \leq \frac{(F_1+F_2)M_1+ (G_1+G_2)\overline{M}_1}{min \big ( 1- \mathcal{K}_1; 1- \mathcal{K}_2 \big)}.
 \end{equation}
 Where $\mathcal{K}_1; \mathcal{K}_2$ are given by $(40)-(41)$. From $(48)$ the set $\mathfrak{M}$ is bounded. Therefore, by applying Theorem $2.7$, the operator $\mathcal{A}$ has at least one fixed point. Therefore, we deduce that
problem $(2)$ has at least one solution on $[a, b]$.
\end{proof} 
\subsection{Uniqueness result via Banach’s fixed point theorem}
 To deal with the existence and uniqueness of solution for our system $(2)$, we use Banach contraction principle.
\begin{theorem}
Suppose that $f, g: [a, b] \times \mathbb{ R} \times \mathbb{R} \longrightarrow \mathbb{R}$ are continuous functions. In addition, we assume that :\\
$(H1)$ There exist constants $\mathcal{L}_1, \mathcal{L}_2 > 0$ such that, for all $t \in [a, b]$ and $x_1, x_2, y_1, y_2 \in \mathbb{R}$,
\begin{equation}
\vert f(t,x_2,y_2)-f(t,x_1,y_1) \vert \leq \mathcal{L}_1  \Big( \vert x_2- x_1 \vert + \vert y_2-y_1 \vert\Big ),
\end{equation}
\begin{equation}
\vert g(t,x_2,y_2)-g(t,x_1,y_1) \vert \leq  \mathcal{L}_2 \Big (\vert x_2- x_1 \vert + \vert y_2-y_1 \vert \Big).
\end{equation}
Then, the problem $(1)$ has a unique solution on $[a, b]$, if
\begin{equation}
\Big ( F_1+F_2\Big )\mathcal{L}_1+\Big (G_1+G_2 \Big )\mathcal{L}_2+\Big ( X_1+X_2\Big )+\Big (Y_1+Y_2 \Big ) <1.
\end{equation}
\end{theorem}
\begin{proof}
Consider the operator $\mathcal{A}$ defined in $(27)$. The system $(2)$ is then transformed into a fixed point problem  $(x, y)(t)=\mathcal{A}(x, y)(t)$. By using  Banach contraction principle we will show that $\mathcal{A}$ has a unique fixed point.\\
We set $sup_{t\in [a, b]}=\vert f(t,0,0)\vert=L_1 < \infty$, and $sup_{t\in [a, b]}=\vert g(t,0,0)\vert=L_2 < \infty$ and choose $r > 0$ such that
\begin{equation}
r\geq \frac{(F_1+F_2)L_1+(G_1+G_2)L_2}{1-\Big [ \Big ( F_1+F_2\Big )\mathcal{L}_1+\Big (G_1+G_2 \Big )\mathcal{L}_2+\Big ( X_1+X_2\Big )+\Big (Y_1+Y_2 \Big )\Big] }.
\end{equation}
Now, we show that $\mathcal{A}\mathcal{B}_r \subset \mathcal{B}_r$, where $\mathcal{B}_r=\lbrace (x, y)\in \mathcal{K}\times \mathcal{K} : \Vert (x,y)\Vert\leq r \rbrace$. for any $(x, y)\in\mathcal{B}_r $, $t\in [a, b]$ we have 
\begin{align}
\vert f(t,x(t),y(t)) \vert  &\leq \vert f(t,x(t),y(t))-f(t,0 ,0) \vert + \vert f(t,0 ,0)\vert\nonumber \\
 & \leq \mathcal{L}_1 (\vert x(t) \vert+\vert y(t) \vert)+L_1\nonumber\\
 & \leq \mathcal{L}_1 (\Vert x \Vert+\Vert y \Vert)+L_1 \nonumber\\
 &\leq \mathcal{L}_1 r+L_1,
  \end{align}
 similarly, we have
  \begin{equation}
  \vert g(t,x(t),y(t)) \vert  \leq \mathcal{L}_2 (\Vert x \Vert+\Vert y \Vert)+L_2 \leq \mathcal{L}_2 r+L_2.
  \end{equation}
Then we get 
\begin{align*}
\big \vert &\mathcal{A}_1(x,y)(t) \big \vert\leq \underset{t \in [a, b ]}{sup} \Bigg \lbrace I^{\alpha_1 + \alpha_2}\vert f(t, x(t), y(t))\vert+ \vert \lambda_1 \vert I^{\alpha_2}\vert x(t)\vert +\displaystyle \frac{(t-a)^{\gamma_1+\alpha_2-1}}{\vert \Lambda \vert \Gamma(\gamma_1+\alpha_2)}\\
&\times\Bigg [\vert \Phi_4\vert \Bigg (  I^{\alpha_1 + \alpha_2}\vert f(b, x(b), y(b))\vert+ \displaystyle \sum\limits_{i=1}^{n} \vert \mu_i\vert I^{p_1 + q_2+\nu_i}\vert g(\eta_i, x(\eta_i), y(\eta_i))\vert \\
&+\vert \lambda_1 \vert I^{\alpha_2}\vert x(b)\vert+\vert \lambda_2 \vert \displaystyle \sum\limits_{i=1}^{n} \vert \mu_i\vert  I^{p_2+\nu_i}\vert y(\eta_i)\vert \Bigg )\\
& +\vert \Phi_2\vert \Bigg ( \displaystyle \sum\limits_{j=1}^{m} \vert \omega_j\vert I^{\alpha_1 + \alpha_2+\sigma_j}\vert f(\xi_j, x(\xi_j), y(\xi_j))\vert+I^{p_1 + p_2}\vert g(b, x(b), y(b))\vert 
\end{align*}
\begin{align*}
+&\displaystyle \sum\limits_{j=1}^{m} \vert \omega_j\vert I^{\alpha_2+\sigma_j}\vert x(\xi_j)\vert + \vert \lambda_2 \vert I^{p_2}\vert y(b)\vert \Bigg) \Bigg ]\Bigg \rbrace \\
&\leq\displaystyle\frac{(b-a)^{\alpha_1 +\alpha_2}}{\Gamma (\alpha_1 +\alpha_2+1)}\Big ( \mathcal{L}_1 r+L_1\Big) + \vert \lambda_1 \vert \frac{(b-a)^{\alpha_2}}{\Gamma (\alpha_2+1)}\Vert x \Vert+ \frac{(b-a)^{\gamma_1+\alpha_2-1}}{\vert \Lambda \vert \Gamma(\gamma_1+\alpha_2)} \\
&\times \Bigg [ \vert \Phi_4\vert \Bigg (\frac{(b-a)^{\alpha_1 +\alpha_2}}{\Gamma (\alpha_1 +\alpha_2+1)}\Big ( \mathcal{L}_1 r+L_1\Big) + \displaystyle \sum\limits_{i=1}^{n} \vert \mu_i \vert \frac{(\eta_i-a)^{p_1 +p_2+\nu_i}}{\Gamma(p_1 +p_2+\nu_i+1)}\Big ( \mathcal{L}_2 r+L_2\Big)\\
& + \vert \lambda_1 \vert  \frac{(b-a)^{\alpha_2}}{\Gamma (\alpha_2+1)}\Vert x \Vert+\vert \lambda_2\vert \displaystyle \sum\limits_{i=1}^{n} \vert \mu_i \vert \frac{(\eta_i-a)^{p_2+\nu_i}}{\Gamma(p_2+\nu_i+1)}\Vert y \Vert \Bigg)\\
&+ \vert\Phi_2 \vert \Bigg ( \displaystyle \sum\limits_{j=1}^{m} \vert \omega_j\vert\frac{(\xi_j-a)^{\alpha_1+\alpha_2+\sigma_j}}{\Gamma (\alpha_1+\alpha_2+\sigma_j+1)}\Big ( \mathcal{L}_1 r+L_1\Big)+\frac{(b-a)^{p_1 +p_2}}{\Gamma (p_1 +p_2+1)}\Big ( \mathcal{L}_2 r+L_2\Big)\\
&+\vert \lambda_1\vert \displaystyle \sum\limits_{j=1}^{m} \vert \omega_j\vert\frac{(\xi_j-a)^{\alpha_2+\sigma_j}}{\Gamma (\alpha_2+\sigma_j+1)} \Vert x \Vert + \vert \lambda_2\vert \frac{(b-a)^{p_2}}{p_2+1}\Vert y \Vert \Bigg) \Bigg]  \\
&\leq \Big (F_1\mathcal{L}_1 +G_1\mathcal{L}_2 +X_1+Y_1\Big) r+F_1L_1+G_1L_2,
\end{align*}
which implies 
\begin{equation}
 \Vert \mathcal{A}_1(x,y) \big \Vert \leq \Big (F_1\mathcal{L}_1 +G_1\mathcal{L}_2 +X_1+Y_1\Big) r+F_1L_1+G_1L_2.
\end{equation}
Similarly, we find that
\begin{equation}
 \Vert \mathcal{A}_2(x,y) \big \Vert \leq \Big (F_2\mathcal{L}_1 +G_2\mathcal{L}_2 +X_2+Y_2\Big) r+F_2L_1+G_2L_2.
\end{equation}
Then, from $(52)$ we obtain
\begin{align}
 \Vert \mathcal{A}(x,y) \big \Vert \leq &\Big [ \Big ( F_1+F_2\Big )\mathcal{L}_1+\Big (G_1+G_2 \Big )\mathcal{L}_2+\Big ( X_1+X_2\Big )+\Big (Y_1+Y_2 \Big )\Big] r\nonumber\\
 &+(F_1+F_2)L_1+(G_1+G_2)L_2 \leq r, 
\end{align}
which implies that $\mathcal{A}\mathcal{B}_r \subset \mathcal{B}_r$.\\
Next, we will show that the operator $\mathcal{A}$ is a contraction, 
\begin{align*}
\big \vert &\mathcal{A}_1(x_2,y_2)(t)-\mathcal{A}_1(x_1,y_1)(t) \big \vert\leq  I^{\alpha_1 + \alpha_2}\vert f(t, x_2(t), y_2(t))-f(t, x_1(t), y_1(t))\vert\\
&+ \vert \lambda_1 \vert I^{\alpha_2}\vert x_2(t)-x_1(t)\vert +\displaystyle \frac{(t-a)^{\gamma_1+\alpha_2-1}}{\vert \Lambda \vert \Gamma(\gamma_1+\alpha_2)}\\
&\times\Bigg [\vert \Phi_4\vert \Bigg (  I^{\alpha_1 + \alpha_2}\vert f(b, x_2(b), y_2(b))-f(b, x_1(b), y_1(b))\vert \\
&+ \displaystyle \sum\limits_{i=1}^{n} \vert \mu_i\vert I^{p_1 + q_2+\nu_i}\vert g(\eta_i, x_2(\eta_i), y_2(\eta_i))-g(\eta_i, x_1(\eta_i), y_1(\eta_i))\vert\\
&+\vert \lambda_1 \vert I^{\alpha_2}\vert x_2(b)-x_1(b)\vert+\vert \lambda_2 \vert \displaystyle \sum\limits_{i=1}^{n} \vert \mu_i\vert  I^{p_2+\nu_i}\vert y_2(\eta_i)-y_1(\eta_i)\vert \Bigg )\\
& +\vert \Phi_2\vert \Bigg ( \displaystyle \sum\limits_{j=1}^{m} \vert \omega_j\vert I^{\alpha_1 + \alpha_2+\sigma_j}\vert f(\xi_j, x_2(\xi_j), y_2(\xi_j))-f(\xi_j, x_1(\xi_j), y_1(\xi_j))\vert \\
&+I^{p_1 + p_2}\vert g(b, x_2(b), y_2(b))-g(b, x_1(b), y_1(b))\vert \\
&+\displaystyle \sum\limits_{j=1}^{m} \vert \omega_j\vert I^{\alpha_2+\sigma_j}\vert x_2(\xi_j)-x_1(\xi_j)\vert + \vert \lambda_2 \vert I^{p_2}\vert y_2(b)-y_1(b)\vert \Bigg) \Bigg ] \\
&\leq\displaystyle\frac{(b-a)^{\alpha_1 +\alpha_2}}{\Gamma (\alpha_1 +\alpha_2+1)} \mathcal{L}_1 \Big ( \Vert x_2-x_1\Vert +\Vert y_2-y_1 \Vert\Big) + \vert \lambda_1 \vert \frac{(b-a)^{\alpha_2}}{\Gamma (\alpha_2+1)}\Vert x_2 -x_1 \Vert \\
&+ \frac{(b-a)^{\gamma_1+\alpha_2-1}}{\vert \Lambda \vert \Gamma(\gamma_1+\alpha_2)}\times \Bigg [ \vert \Phi_4\vert \Bigg (\frac{(b-a)^{\alpha_1 +\alpha_2}}{\Gamma (\alpha_1 +\alpha_2+1)}\mathcal{L}_1 \Big ( \Vert x_2-x_1\Vert +\Vert y_2-y_1 \Vert\Big)\\
&+ \displaystyle \sum\limits_{i=1}^{n} \vert \mu_i \vert \frac{(\eta_i-a)^{p_1 +p_2+\nu_i}}{\Gamma(p_1 +p_2+\nu_i+1)}\mathcal{L}_2 \Big ( \Vert x_2-x_1\Vert +\Vert y_2-y_1 \Vert\Big)\\
& + \vert \lambda_1 \vert  \frac{(b-a)^{\alpha_2}}{\Gamma (\alpha_2+1)}\Vert x_2- x_1 \Vert+\vert \lambda_2\vert \displaystyle \sum\limits_{i=1}^{n} \vert \mu_i \vert \frac{(\eta_i-a)^{p_2+\nu_i}}{\Gamma(p_2+\nu_i+1)}\Vert y_2 -y_1 \Vert \Bigg)\\
&+ \vert\Phi_2 \vert \Bigg ( \displaystyle \sum\limits_{j=1}^{m} \vert \omega_j\vert\frac{(\xi_j-a)^{\alpha_1+\alpha_2+\sigma_j}}{\Gamma (\alpha_1+\alpha_2+\sigma_j+1)}\mathcal{L}_1 \Big ( \Vert x_2-x_1\Vert \\
&+\Vert y_2-y_1 \Vert\Big)+\frac{(b-a)^{p_1 +p_2}}{\Gamma (p_1 +p_2+1)}\mathcal{L}_2 \Big ( \Vert x_2-x_1\Vert +\Vert y_2-y_1 \Vert\Big)\\
&+\vert \lambda_1\vert \displaystyle \sum\limits_{j=1}^{m} \vert \omega_j\vert\frac{(\xi_j-a)^{\alpha_2+\sigma_j}}{\Gamma (\alpha_2+\sigma_j+1)} \Vert x_2 - x_1 \Vert + \vert \lambda_2\vert \frac{(b-a)^{p_2}}{p_2+1}\Vert y_2-y_1 \Vert \Bigg) \Bigg] \Bigg\rbrace \\
&\leq \Big (F_1\mathcal{L}_1 +G_1\mathcal{L}_2 +X_1+Y_1\Big)\Big ( \Vert x_2-x_1\Vert +\Vert y_2-y_1 \Vert\Big),
\end{align*}
which implies
\begin{equation}
\big \Vert\mathcal{A}_1(x_2,y_2)(t)-\mathcal{A}_1(x_1,y_1)(t) \big \Vert\leq \Big (F_1\mathcal{L}_1 +G_1\mathcal{L}_2 +X_1+Y_1\Big)\Big ( \Vert x_2-x_1\Vert +\Vert y_2-y_1 \Vert\Big).
\end{equation}
Similarly, we find that
 \begin{equation}
\big \Vert\mathcal{A}_2(x_2,y_2)(t)-\mathcal{A}_1(x_1,y_1)(t) \big \Vert\leq \Big (F_2\mathcal{L}_1 +G_2\mathcal{L}_2 +X_2+Y_2\Big)\Big ( \Vert x_2-x_1\Vert +\Vert y_2-y_1 \Vert\Big).
\end{equation}
From $(58)$ and $(59)$ we obtain 
\begin{align}
 \Vert \mathcal{A}(x_2,y_2)-\mathcal{A}(x_1,y_1) \big \Vert \leq &\Big [ \Big ( F_1+F_2\Big )\mathcal{L}_1+\Big (G_1+G_2 \Big )\mathcal{L}_2+\Big ( X_1+X_2\Big )\nonumber\\
 &+\Big (Y_1+Y_2 \Big )\Big]\Big ( \Vert x_2-x_1\Vert +\Vert y_2-y_1 \Vert\Big).
\end{align}
As $ \Big ( F_1+F_2\Big )\mathcal{L}_1+\Big (G_1+G_2 \Big )\mathcal{L}_2+\Big ( X_1+X_2\Big )+\Big (Y_1+Y_2 \Big )<1$, then $\mathcal{A}$  is a contraction operator. Therefore, by Banach’s fixed-point theorem, the operator $\mathcal{A}$ has a unique fixed point which is indeed a unique solution of system $(2)$ on $[a, b]$. The proof is completed. 
\end{proof}
\section{ Ulam–Hyers stability analysis}

In this section, we study Ulam–Hyers (U-H), generalized Ulam–Hyers (G-U-H), stability of solution to the Hilfer coupled system $(2)$ .\\

Let $\varepsilon=(\varepsilon_1, \varepsilon_2)>0$, we consider the following inequalities
\begin{equation}
\Big \vert \displaystyle ^{H}D^{\alpha_1 , \beta_1} ( ^{H}D^{\alpha_2 , \beta_2}+\lambda_1  )\tilde x(t)- f(t,\tilde x(t),\tilde y(t))\Big \vert \leq \varepsilon_1 , \quad t\in [a, b],
\end{equation}
\begin{equation}
\Big \vert \displaystyle ^{H}D^{p_1 , q_1} ( ^{H}D^{p_2 , q_2}+\lambda_2  )\tilde y(t)- g(t,\tilde x(t),\tilde y(t))\Big \vert\leq \varepsilon_2, \quad t\in [a, b],
\end{equation}
and $\tilde{x}(b)=x(b)$, $\tilde{y}(b)=y(b)$
\begin{definition} \cite{ref1, ref7}
The Hilfer coupled system is U-H stable if there exists $\lambda=(\lambda_f, \lambda_g)>0$, such that for each $\varepsilon=(\varepsilon_1, \varepsilon_2)>0$ and for each solution $(\tilde x ,\tilde y)\in \mathcal{K}\times \mathcal{K}$ of inequalities $(61)$, $(62)$, there exists $(x, y)\in \mathcal{K}\times \mathcal{K}$ solution of the coupled system $(2)$ complying with 
\begin{equation}
\Vert  (\tilde x ,\tilde y)-(x, y) \Vert _{\mathcal{K}\times \mathcal{K}}\leq \lambda \varepsilon.
\end{equation}
\end{definition}
\begin{definition} \cite{ref1, ref7}
The Hilfer coupled system is G-U-H stable if there exists $\varphi=(\varphi_f, \varphi_g)\in \mathcal{C}(\mathbb{R}, \mathbb{R}))$ with $\varphi(0)=(\varphi_f(0), \varphi_g(0))=(0, 0)$ , such that for each $\varepsilon=(\varepsilon_1, \varepsilon_2)>0$ and for each solution $(\tilde x ,\tilde y)\in \mathcal{K}\times \mathcal{K}$ of inequalities $(60)-(61)$, there exists $(x, y)\in \mathcal{K}\times \mathcal{K}$ solution of the coupled system $(2)$ complying with 
\begin{equation}
\Vert  (\tilde x ,\tilde y)-(x, y) \Vert _{\mathcal{K}\times \mathcal{K}}\leq \varphi(\varepsilon).
\end{equation}
\end{definition}

\begin{remark}
A function $(\tilde x ,\tilde y)\in \mathcal{K}\times \mathcal{K}$  is a solution of inequalities $(61)-(62)$ if and only if
there exists a function $(h_1 ,h_2)\in \mathcal{K}\times \mathcal{K}$ such that
\begin{itemize}
\item[i-]$\vert h_1(t) \vert \leq \varepsilon_1$ and $\vert h_2(t) \vert \leq \varepsilon_2$,
\item[ii-] for $t \in [a, b]$
\begin{equation}
	\begin{cases}
	\displaystyle ^{H}D^{\alpha_1 , \beta_1} ( ^{H}D^{\alpha_2 , \beta_2}+\lambda_1  )\tilde x(t)= f(t,\tilde x(t),\tilde y(t))+h_1(t) \\
	\displaystyle ^{H}D^{p_1 , q_1} ( ^{H}D^{p_2 , q_2}+\lambda_2  )\tilde y(t)= g(t,\tilde x(t),\tilde y(t)) +h_2(t). 
	\end{cases}
\end{equation}
\end{itemize}
\end{remark}
To simplify the computations, we use the following notations:
\begin{align}
A_1&=\displaystyle\frac{(b-a)^{\alpha_1 +\alpha_2}}{\Gamma (\alpha_1 +\alpha_2+1)} \mathcal{L}_1+\vert \lambda_1 \vert \frac{(b-a)^{\alpha_2}}{\Gamma (\alpha_2+1)}+\frac{(b-a)^{\gamma_1+\alpha_2-1}}{\vert \Lambda \vert \Gamma(\gamma_1+\alpha_2)}\nonumber\\
&\times\Bigg [\vert \Phi_4\vert\displaystyle \sum\limits_{i=1}^{n} \vert \mu_i \vert \frac{(\eta_i-a)^{p_1 +p_2+\nu_i}}{\Gamma(p_1 +p_2+\nu_i+1)}\mathcal{L}_2 +\vert\Phi_2 \vert \Bigg ( \displaystyle \sum\limits_{j=1}^{m} \vert \omega_j\vert\frac{(\xi_j-a)^{\alpha_1+\alpha_2+\sigma_j}}{\Gamma (\alpha_1+\alpha_2+\sigma_j+1)}\mathcal{L}_1\nonumber\\
&+\vert \lambda_1\vert \displaystyle \sum\limits_{j=1}^{m} \vert \omega_j\vert\frac{(\xi_j-a)^{\alpha_2+\sigma_j}}{\Gamma (\alpha_2+\sigma_j+1)}\Bigg) \Bigg],
\end{align}
\begin{align}
B_1&=\displaystyle\frac{(b-a)^{\alpha_1 +\alpha_2}}{\Gamma (\alpha_1 +\alpha_2+1)} \mathcal{L}_1+\frac{(b-a)^{\gamma_1+\alpha_2-1}}{\vert \Lambda \vert \Gamma(\gamma_1+\alpha_2)}\nonumber\\
&\times\Bigg [\vert \Phi_4\vert \Bigg ( \vert \lambda_2\vert \displaystyle \sum\limits_{i=1}^{n} \vert \mu_i \vert \frac{(\eta_i-a)^{p_2+\nu_i}}{\Gamma(p_2+\nu_i+1)}+\displaystyle \sum\limits_{i=1}^{n} \vert \mu_i \vert \frac{(\eta_i-a)^{p_1 +p_2+\nu_i}}{\Gamma(p_1 +p_2+\nu_i+1)}\mathcal{L}_2 \Bigg)\nonumber \\
& +\vert\Phi_2 \vert  \displaystyle \sum\limits_{j=1}^{m} \vert \omega_j\vert\frac{(\xi_j-a)^{\alpha_1+\alpha_2+\sigma_j}}{\Gamma (\alpha_1+\alpha_2+\sigma_j+1)}\mathcal{L}_1 \Bigg],
\end{align}
\begin{align}
A_2&=\displaystyle\frac{(b-a)^{p_1 +p_2}}{\Gamma (p_1 +p_2+1)} \mathcal{L}_2+\vert \lambda_2 \vert \frac{(b-a)^{p_2}}{\Gamma (p_2+1)}+\frac{(b-a)^{\delta_1+p_2-1}}{\vert \Lambda \vert \Gamma(\delta_1+p_2)}\nonumber\\
&\times\Bigg [\vert \Phi_1\vert\displaystyle \sum\limits_{j=1}^{m} \vert \omega_j \vert \frac{(\xi_j-a)^{\alpha_1 +\alpha_2+\sigma_j}}{\Gamma(\alpha_1 +\alpha_2+\sigma_j+1)}\mathcal{L}_1 +\vert\Phi_3 \vert \Bigg ( \displaystyle \sum\limits_{i=1}^{n} \vert \mu_i \vert\frac{(\eta_i-a)^{p_1+p_2+\nu_i}}{\Gamma (p_1+p_2+\nu_i+1)}\mathcal{L}_2\nonumber\\
&+\vert \lambda_2\vert \displaystyle \sum\limits_{i=1}^{n} \vert \mu_i \vert \frac{(\eta_i-a)^{p_2+\nu_i}}{\Gamma(p_2+\nu_i+1)}\Bigg) \Bigg],
\end{align}
\begin{align}
B_2&=\displaystyle\frac{(b-a)^{p_1 +p_2}}{\Gamma (p_1 +p_2+1)} \mathcal{L}_2+\frac{(b-a)^{\delta_1+p_2-1}}{\vert \Lambda \vert \Gamma(\delta_1+p_2)}\nonumber\\
&\times\Bigg [\vert \Phi_1\vert \Bigg ( \vert \lambda_1\vert \displaystyle \sum\limits_{j=1}^{m} \vert \omega_j \vert \frac{(\xi_j-a)^{\alpha_2+\sigma_j}}{\Gamma(\alpha_2+\sigma_j+1)}+\displaystyle \sum\limits_{j=1}^{m} \vert \omega_j \vert \frac{(\xi_j-a)^{\alpha_1 +\alpha_2+\sigma_j}}{\Gamma(\alpha_1 +\alpha_2+\sigma_j+1)}\mathcal{L}_1 \Bigg)\nonumber \\
& +\vert\Phi_3 \vert  \displaystyle \sum\limits_{i=1}^{n} \vert \mu_i \vert\frac{(\eta_i-a)^{p_1+p_2+\nu_i}}{\Gamma (p_1+p_2+\nu_i+1)}\mathcal{L}_2 \Bigg],
\end{align}
\\
\\
\begin{theorem}
Assume that $(H1)$ hold,if $A_1>1, A_2>1$, and $1- \displaystyle \frac{B_1B_2}{(1-A_1)(1-A_2)}\neq 0$ then the system $(2)$ is Ulam–Hyers stable on $[a, b]$ and consequently generalized Ulam–Hyers stable, where $A_i, B_i, i=1,2$ are given by $(66)-(69)$.
\end{theorem}
\begin{proof}
Let $\varepsilon=(\varepsilon_1, \varepsilon_2)>0$, and $(\tilde x ,\tilde y)\in \mathcal{K}\times \mathcal{K}$ satisfies inequalities $(61)-(62)$, and $(x, y)\in \mathcal{K}\times \mathcal{K}$ be the unique solution of the problem $(2)$ with the conditions $\tilde{x}(b)=x(b)$, $\tilde{y}(b)=y(b)$, then by Lemma $2.6$, we have 
 \begin{align}
x(t)&=I^{\alpha_1 + \alpha_2}f(t,x(t),y(t))- \lambda_1 I^{\alpha_2}x(t)+\displaystyle \frac{(t-a)^{\gamma_1+\alpha_2-1}}{\Lambda \Gamma(\gamma_1+\alpha_2)}\nonumber \\
&\times\Bigg [ \Phi_4 \Bigg ( -I^{\alpha_1 + \alpha_2}f(b,x(b),y(b)))-\displaystyle \sum\limits_{i=1}^{n} \mu_i I^{p_1 + p_2+\nu_i}g(\eta_i,x(\eta_i),y(\eta_i))\nonumber\\
& +\lambda_1 I^{\alpha_2}x(b)-\lambda_2 \displaystyle \sum\limits_{i=1}^{n} \mu_i I^{p_2+\nu_i}y(\eta_i)\Bigg )+ \Phi_2 \Bigg ( \displaystyle \sum\limits_{j=1}^{m} \omega_j I^{\alpha_1 + \alpha_2+\sigma_j}f(\xi_j,x(\xi_j),y(\xi_j)) \nonumber \\
& -I^{p_1 + p_2}g(b,x(b),y(b))-\displaystyle \sum\limits_{j=1}^{m} \omega_j I^{\alpha_2+\sigma_j}x(\xi_j)+ \lambda_2 I^{p_2}y(b)\Bigg) \Bigg],
\end{align}
and 
\begin{align}
y(t)&=I^{p_1 + p_2}g(t,x(t),y(t))- \lambda_2 I^{p_2}y(t)+\displaystyle \frac{(t-a)^{\delta_1+p_2-1}}{\Lambda \Gamma(\delta_1+p_2)}\nonumber\\
&\times \Bigg [ \Phi_1 \Bigg (\displaystyle \sum\limits_{j=1}^{m} \omega_j I^{\alpha_1 + \alpha_2+\sigma_j}h_1(\xi_j)-I^{p_1 +p_2}g(b,x(b),y(b))\nonumber\\
&- \lambda_1 \displaystyle \sum\limits_{j=1}^{m} \omega_j I^{\alpha_2+\sigma_j}x(\xi_j) +\lambda_2 I^{p_2}y(b) \Bigg ) +\displaystyle \Phi_3 \Bigg ( - I^{\alpha_1 + \alpha_2}f(b,x(b),y(b))\nonumber \\
&+ \displaystyle \sum\limits_{i=1}^{n} \mu_i I^{p_1+p_2+\nu_i}g(\eta_i,x(\eta_i),y(\eta_i))+\lambda_1 I^{\alpha_2}x(b)-\lambda_2 \displaystyle \sum\limits_{i=1}^{n} \mu_i I^{p_2+\nu_i}y(\eta_i) \Bigg)\Bigg].
\end{align}
Since, $(\tilde x ,\tilde y)\in \mathcal{K}\times \mathcal{K}$ satisfies inequalities $(61)-(62)$, by the remark we have  
\begin{equation}
	\begin{cases}
	\displaystyle ^{H}D^{\alpha_1 , \beta_1} ( ^{H}D^{\alpha_2 , \beta_2}+\lambda_1  )\tilde x(t)= f(t,\tilde x(t),\tilde y(t))+h_1(t) , \quad t\in [a,b],\\
	\displaystyle ^{H}D^{p_1 , q_1} ( ^{H}D^{p_2 , q_2}+\lambda_2  )\tilde y(t)= g(t,\tilde x(t),\tilde y(t)) +h_2(t) , \quad t\in [a,b], \\
	\tilde{x}(a)=x(a) \quad , \quad \tilde{x}(b)=x(b),\\
	\tilde{y}(a)=y(a) \quad , \quad \tilde{y}(b)=y(b),
	 
	\end{cases}
\end{equation}
then by Lemma $2.6$, we have 

\begin{align}
\tilde x(t)&=I^{\alpha_1 + \alpha_2}f(t,\tilde x(t),\tilde y(t))- \lambda_1 I^{\alpha_2}\tilde x(t)+\displaystyle \frac{(t-a)^{\gamma_1+\alpha_2-1}}{\Lambda \Gamma(\gamma_1+\alpha_2)}\nonumber \\
&\times\Bigg [ \Phi_4 \Bigg ( -I^{\alpha_1 + \alpha_2}f(b,\tilde x(b),\tilde y(b)))-\displaystyle \sum\limits_{i=1}^{n} \mu_i I^{p_1 + p_2+\nu_i}g(\eta_i,\tilde x(\eta_i),\tilde y(\eta_i))\nonumber
\end{align}
\begin{align}
 +&\lambda_1 I^{\alpha_2}\tilde x(b)-\lambda_2 \displaystyle \sum\limits_{i=1}^{n} \mu_i I^{p_2+\nu_i}\tilde y(\eta_i)\Bigg )+ \Phi_2 \Bigg ( \displaystyle \sum\limits_{j=1}^{m} \omega_j I^{\alpha_1 + \alpha_2+\sigma_j}f(\xi_j,\tilde x(\xi_j),\tilde y(\xi_j)) \nonumber \\
& -I^{p_1 + p_2}g(b,\tilde x(b),\tilde y(b))-\displaystyle \sum\limits_{j=1}^{m} \omega_j I^{\alpha_2+\sigma_j}\tilde x(\xi_j)+ \lambda_2 I^{p_2}\tilde y(b)\Bigg) \Bigg]+I^{\alpha_1 + \alpha_2}h_1(t),
\end{align}
and 
\begin{align}
\tilde y(t)&=I^{p_1 + p_2}g(t,\tilde x(t),\tilde y(t))- \lambda_2 I^{p_2}\tilde y(t)+\displaystyle \frac{(t-a)^{\delta_1+p_2-1}}{\Lambda \Gamma(\delta_1+p_2)}\nonumber\\
&\times \Bigg [ \Phi_1 \Bigg (\displaystyle \sum\limits_{j=1}^{m} \omega_j I^{\alpha_1 + \alpha_2+\sigma_j}f(\xi_j, \tilde x(\xi_j), \tilde y(\xi_j))-I^{p_1 +p_2}g(b,\tilde x(b),\tilde y(b))\nonumber\\
&- \lambda_1 \displaystyle \sum\limits_{j=1}^{m} \omega_j I^{\alpha_2+\sigma_j}\tilde x(\xi_j) +\lambda_2 I^{p_2}\tilde y(b) \Bigg ) +\displaystyle \Phi_3 \Bigg ( - I^{\alpha_1 + \alpha_2}f(b,\tilde x(b),\tilde y(b))\nonumber \\
&+ \displaystyle \sum\limits_{i=1}^{n} \mu_i I^{p_1+p_2+\nu_i}g(\eta_i,\tilde x(\eta_i),\tilde y(\eta_i))+\lambda_1 I^{\alpha_2}\tilde x(b)-\lambda_2 \displaystyle \sum\limits_{i=1}^{n} \mu_i I^{p_2+\nu_i}\tilde y(\eta_i) \Bigg)\Bigg]\nonumber\\
&+I^{p_1 + p_2}h_2(t).
\end{align}
On the other hand, we have, for each $t \in [a, b]$
\begin{align*}
\big \vert &\tilde x(t) - x(t) \big \vert\leq  I^{\alpha_1 + \alpha_2}\vert f(t,\tilde x(t),\tilde y(t))-f(t, x(t), y(t))\vert\\
&+ \vert \lambda_1 \vert I^{\alpha_2}\vert \tilde x(t)-x(t)\vert +\displaystyle \frac{(t-a)^{\gamma_1+\alpha_2-1}}{\vert \Lambda \vert \Gamma(\gamma_1+\alpha_2)}\\
&\times\Bigg [\vert \Phi_4\vert \Bigg ( \displaystyle \sum\limits_{i=1}^{n} \vert \mu_i\vert I^{p_1 + q_2+\nu_i}\vert g(\eta_i, \tilde x(\eta_i), \tilde y(\eta_i))-g(\eta_i, x(\eta_i), y(\eta_i))\vert  \\
&+\vert \lambda_2 \vert \displaystyle \sum\limits_{i=1}^{n} \vert \mu_i\vert  I^{p_2+\nu_i}\vert \tilde y(\eta_i)-y(\eta_i)\vert \Bigg )\\
& +\vert \Phi_2\vert \Bigg ( \displaystyle \sum\limits_{j=1}^{m} \vert \omega_j\vert I^{\alpha_1 + \alpha_2+\sigma_j}\vert f(\xi_j, \tilde x(\xi_j), \tilde y(\xi_j))-f(\xi_j, x(\xi_j), y(\xi_j))\vert \\
&+\vert \lambda_1 \vert\displaystyle \sum\limits_{j=1}^{m} \vert \omega_j\vert I^{\alpha_2+\sigma_j}\vert \tilde  x(\xi_j)-x(\xi_j)\vert  \Bigg) \Bigg ]+ I^{\alpha_1 + \alpha_2} \vert h_1(t) \vert\\
&\leq\displaystyle\frac{(b-a)^{\alpha_1 +\alpha_2}}{\Gamma (\alpha_1 +\alpha_2+1)} \mathcal{L}_1 \Big ( \Vert \tilde x-x\Vert_\mathcal{K} +\Vert \tilde y-y \Vert_\mathcal{K}\Big) + \vert \lambda_1 \vert \frac{(b-a)^{\alpha_2}}{\Gamma (\alpha_2+1)}\Vert \tilde x -x \Vert_\mathcal{K} \\
&+ \frac{(b-a)^{\gamma_1+\alpha_2-1}}{\vert \Lambda \vert \Gamma(\gamma_1+\alpha_2)}\times \Bigg [ \vert \Phi_4\vert \Bigg (\vert \lambda_2\vert \displaystyle \sum\limits_{i=1}^{n} \vert \mu_i \vert \frac{(\eta_i-a)^{p_2+\nu_i}}{\Gamma(p_2+\nu_i+1)}\Vert \tilde y -y \Vert_\mathcal{K} \\
&+ \displaystyle \sum\limits_{i=1}^{n} \vert \mu_i \vert \frac{(\eta_i-a)^{p_1 +p_2+\nu_i}}{\Gamma(p_1 +p_2+\nu_i+1)}\mathcal{L}_2 \Big ( \Vert \tilde x-x\Vert_\mathcal{K} +\Vert \tilde y-y \Vert_\mathcal{K}\Big)\Bigg)\\
&+ \vert\Phi_2 \vert \Bigg ( \displaystyle \sum\limits_{j=1}^{m} \vert \omega_j\vert\frac{(\xi_j-a)^{\alpha_1+\alpha_2+\sigma_j}}{\Gamma (\alpha_1+\alpha_2+\sigma_j+1)}\mathcal{L}_1 \Big ( \Vert \tilde x-x\Vert_\mathcal{K} 
\end{align*}
\begin{align*}
+&\Vert \tilde y-y \Vert_\mathcal{K}\Big)+\vert \lambda_1\vert \displaystyle \sum\limits_{j=1}^{m} \vert \omega_j\vert\frac{(\xi_j-a)^{\alpha_2+\sigma_j}}{\Gamma (\alpha_2+\sigma_j+1)} \Vert \tilde x - x \Vert_\mathcal{K}  \Bigg) \Bigg]+\displaystyle\frac{(b-a)^{\alpha_1 +\alpha_2}}{\Gamma (\alpha_1 +\alpha_2+1)}\varepsilon_1\\
&\leq A_1\Vert \tilde x - x \Vert_\mathcal{K} +B_1\Vert \tilde y-y \Vert_\mathcal{K}+C_1\varepsilon_1,
\end{align*}
which implies,
\begin{equation}
\Vert \tilde x - x \Vert_\mathcal{K} \leq \frac{C_1}{1-A_1}\varepsilon_1 + \frac{B_1}{1-A_1}\Vert \tilde y-y \Vert_\mathcal{K}.
\end{equation}
Where 
\begin{align}
A_1&=\displaystyle\frac{(b-a)^{\alpha_1 +\alpha_2}}{\Gamma (\alpha_1 +\alpha_2+1)} \mathcal{L}_1+\vert \lambda_1 \vert \frac{(b-a)^{\alpha_2}}{\Gamma (\alpha_2+1)}+\frac{(b-a)^{\gamma_1+\alpha_2-1}}{\vert \Lambda \vert \Gamma(\gamma_1+\alpha_2)}\nonumber\\
&\times\Bigg [\vert \Phi_4\vert\displaystyle \sum\limits_{i=1}^{n} \vert \mu_i \vert \frac{(\eta_i-a)^{p_1 +p_2+\nu_i}}{\Gamma(p_1 +p_2+\nu_i+1)}\mathcal{L}_2 +\vert\Phi_2 \vert \Bigg ( \displaystyle \sum\limits_{j=1}^{m} \vert \omega_j\vert\frac{(\xi_j-a)^{\alpha_1+\alpha_2+\sigma_j}}{\Gamma (\alpha_1+\alpha_2+\sigma_j+1)}\mathcal{L}_1\nonumber\\
&+\vert \lambda_1\vert \displaystyle \sum\limits_{j=1}^{m} \vert \omega_j\vert\frac{(\xi_j-a)^{\alpha_2+\sigma_j}}{\Gamma (\alpha_2+\sigma_j+1)}\Bigg) \Bigg],
\end{align}
\begin{align}
B_1&=\displaystyle\frac{(b-a)^{\alpha_1 +\alpha_2}}{\Gamma (\alpha_1 +\alpha_2+1)} \mathcal{L}_1+\frac{(b-a)^{\gamma_1+\alpha_2-1}}{\vert \Lambda \vert \Gamma(\gamma_1+\alpha_2)}\nonumber\\
&\times\Bigg [\vert \Phi_4\vert \Bigg ( \vert \lambda_2\vert \displaystyle \sum\limits_{i=1}^{n} \vert \mu_i \vert \frac{(\eta_i-a)^{p_2+\nu_i}}{\Gamma(p_2+\nu_i+1)}+\displaystyle \sum\limits_{i=1}^{n} \vert \mu_i \vert \frac{(\eta_i-a)^{p_1 +p_2+\nu_i}}{\Gamma(p_1 +p_2+\nu_i+1)}\mathcal{L}_2 \Bigg)\nonumber \\
& +\vert\Phi_2 \vert  \displaystyle \sum\limits_{j=1}^{m} \vert \omega_j\vert\frac{(\xi_j-a)^{\alpha_1+\alpha_2+\sigma_j}}{\Gamma (\alpha_1+\alpha_2+\sigma_j+1)}\mathcal{L}_1 \Bigg],
\end{align}
\begin{equation}
C_1=\displaystyle\frac{(b-a)^{\alpha_1 +\alpha_2}}{\Gamma (\alpha_1 +\alpha_2+1)}.
\end{equation}
Similarly, we have
\begin{equation}
\Vert \tilde y - y \Vert_\mathcal{K} \leq \frac{C_2}{1-A_2}\varepsilon_2 + \frac{B_2}{1-A_2}\Vert \tilde x-x \Vert_\mathcal{K}.
\end{equation}
Where 
\begin{align}
A_2&=\displaystyle\frac{(b-a)^{p_1 +p_2}}{\Gamma (p_1 +p_2+1)} \mathcal{L}_2+\vert \lambda_2 \vert \frac{(b-a)^{p_2}}{\Gamma (p_2+1)}+\frac{(b-a)^{\delta_1+p_2-1}}{\vert \Lambda \vert \Gamma(\delta_1+p_2)}\nonumber\\
&\times\Bigg [\vert \Phi_1\vert\displaystyle \sum\limits_{j=1}^{m} \vert \omega_j \vert \frac{(\xi_j-a)^{\alpha_1 +\alpha_2+\sigma_j}}{\Gamma(\alpha_1 +\alpha_2+\sigma_j+1)}\mathcal{L}_1 +\vert\Phi_3 \vert \Bigg ( \displaystyle \sum\limits_{i=1}^{n} \vert \mu_i \vert\frac{(\eta_i-a)^{p_1+p_2+\nu_i}}{\Gamma (p_1+p_2+\nu_i+1)}\mathcal{L}_2\nonumber\\
&+\vert \lambda_2\vert \displaystyle \sum\limits_{i=1}^{n} \vert \mu_i \vert \frac{(\eta_i-a)^{p_2+\nu_i}}{\Gamma(p_2+\nu_i+1)}\Bigg) \Bigg],
\end{align}
\begin{align}
B_2&=\displaystyle\frac{(b-a)^{p_1 +p_2}}{\Gamma (p_1 +p_2+1)} \mathcal{L}_2+\frac{(b-a)^{\delta_1+p_2-1}}{\vert \Lambda \vert \Gamma(\delta_1+p_2)}\nonumber\\
&\times\Bigg [\vert \Phi_1\vert \Bigg ( \vert \lambda_1\vert \displaystyle \sum\limits_{j=1}^{m} \vert \omega_j \vert \frac{(\xi_j-a)^{\alpha_2+\sigma_j}}{\Gamma(\alpha_2+\sigma_j+1)}+\displaystyle \sum\limits_{j=1}^{m} \vert \omega_j \vert \frac{(\xi_j-a)^{\alpha_1 +\alpha_2+\sigma_j}}{\Gamma(\alpha_1 +\alpha_2+\sigma_j+1)}\mathcal{L}_1 \Bigg)\nonumber \\
& +\vert\Phi_3 \vert  \displaystyle \sum\limits_{i=1}^{n} \vert \mu_i \vert\frac{(\eta_i-a)^{p_1+p_2+\nu_i}}{\Gamma (p_1+p_2+\nu_i+1)}\mathcal{L}_2 \Bigg],
\end{align}
\begin{equation}
C_2=\displaystyle\frac{(b-a)^{p_1 +p_2}}{\Gamma (p_1 +p_2+1)}.
\end{equation}
It follows that 
\begin{equation}
	\begin{cases}
	\displaystyle \Vert \tilde x - x \Vert_\mathcal{K} - \frac{B_1}{1-A_1}\Vert \tilde y-y \Vert_\mathcal{K} \leq \frac{C_1}{1-A_1}\varepsilon_1 .\\
	\displaystyle  \Vert \tilde y - y \Vert_\mathcal{K} - \frac{B_2}{1-A_2}\Vert \tilde x-x \Vert_\mathcal{K}\leq \frac{C_2}{1-A_2}\varepsilon_2 .
	\end{cases}
\end{equation}
Then $(79)$ can be writing as matrices as follows
\begin{equation}
\displaystyle  \begin{pmatrix}
1 & -\frac{B_1}{1-A_1} \\
- \frac{B_2}{1-A_2} & 1 
\end{pmatrix}
\displaystyle  \begin{pmatrix}
\Vert \tilde x - x \Vert_\mathcal{K}\\
 \Vert \tilde y - y \Vert_\mathcal{K}
\end{pmatrix}
\leq
\displaystyle  \begin{pmatrix}
\frac{C_1}{1-A_1}\varepsilon_1\\
\frac{C_2}{1-A_2}\varepsilon_2
\end{pmatrix},
\end{equation}
by simple computations, the above inequality becomes
\begin{equation}
\displaystyle  \begin{pmatrix}
\Vert \tilde x - x \Vert_\mathcal{K}\\
 \Vert \tilde y - y \Vert_\mathcal{K}
\end{pmatrix}
\leq
\displaystyle  \begin{pmatrix}
\frac{1}{\Delta} & \frac{B_1}{\Delta(1-A_1)} \\
 \frac{B_2}{\Delta(1-A_2)} & \frac{1}{\Delta} 
\end{pmatrix}
\times\displaystyle  \begin{pmatrix}
\frac{C_1}{1-A_1}\varepsilon_1\\
\frac{C_2}{1-A_2}\varepsilon_2
\end{pmatrix},
\end{equation}
where $\Delta = 1- \displaystyle \frac{B_1B_2}{(1-A_1)(1-A_2)}\neq 0$, this leads to
\begin{equation}
\displaystyle \Vert \tilde x - x \Vert_\mathcal{K} \leq \frac{C_1}{\Delta(1-A_1)}\varepsilon_1 + \frac{B_1C_2}{\Delta(1-A_1)(1-A_2)} \varepsilon_2.
\end{equation}
\begin{equation}
\displaystyle \Vert \tilde y - y \Vert_\mathcal{K} \leq \frac{B_2C_1}{\Delta(1-A_1)(1-A_2)}\varepsilon_1 + \frac{C_2}{\Delta(1-A_2)} \varepsilon_2.
\end{equation}
We get 
\begin{equation}
\displaystyle \Vert \tilde x - x \Vert_\mathcal{K}+\Vert \tilde y - y \Vert_\mathcal{K} \leq \frac{C_1(1-A_2)+B_2C_1}{\Delta(1-A_1)(1-A_2)}\varepsilon_1 + \frac{C_2(1-A_1)+B_1C_2}{\Delta(1-A_1)(1-A_2)} \varepsilon_2,
\end{equation}
for $\varepsilon=max(\varepsilon_1;\varepsilon_2)$ and $\lambda=\displaystyle \frac{C_1(1-A_2)+B_2C_1+C_2(1-A_1)+B_1C_2}{\Delta(1-A_1)(1-A_2)}$,
we obtain 
\begin{equation}
\Vert  (\tilde x ,\tilde y)-(x, y) \Vert _{\mathcal{K}\times \mathcal{K}}\leq \lambda \varepsilon.
\end{equation}
This proves that the Hilfer coupled system $(2)$, is U-H stable.\\
Moreover, by setting $\varphi(\varepsilon)=\lambda\varepsilon$ with $\varphi(0)=0$ we get 
\begin{equation}
\Vert  (\tilde x ,\tilde y)-(x, y) \Vert _{\mathcal{K}\times \mathcal{K}}\leq \varphi(\varepsilon).
\end{equation}
This shows that the Hilfer coupled system $(2)$ is G-H-U stable.
\end{proof}
\section*{Acknowledgements}
The authors would like to thank the referees for the valuable comments and suggestions that improve the quality of our paper.
\section*{Disclosure statement}
No potential conflict of interest was reported by the author(s).

\end{document}